\documentclass[12pt, twoside, leqno]{article}
\usepackage{latexsym}
\usepackage{amsmath}
\usepackage{amsfonts}
\usepackage{amssymb}
\usepackage{amsthm}
\usepackage{color}
\numberwithin{equation}{section}

\input xy
\xyoption{all} \hfuzz=1.5pt \tolerance=500
\theoremstyle{plain}

\newtheorem{theorem}{\indent Theorem}[section]
\newtheorem{proposition}[theorem]{\indent Proposition}

\newtheorem{corollary}[theorem]{\indent Corollary}

\theoremstyle{definition}
\newtheorem{definition}[theorem]{\indent Definition}
\newtheorem{example}[theorem]{\indent Example}
\newtheorem{remark}[theorem]{\indent Remark}

\newcommand{\cd}{{\cdot}}
\newcommand{\e}{{\varepsilon}}
\newcommand{\ot}{{\otimes}}
\newcommand{\hil}{\stackrel{\cdot}{\otimes}}
\newcommand{\cK}{{\cal F}}
\newcommand{\la}{\langle}
\newcommand{\ra}{\rangle}
\newcommand{\id}{{\bf 1}}
\newcommand{\di}{\diamondsuit}
\newcommand{\q}{\quad}
\newcommand{\qq}{\qquad}
\newcommand{\va}{\varphi}
\newcommand{\rr}{{\cal R}}
\newcommand{\ii}{\infty}
\newcommand{\mt}{\mapsto}
\newcommand{\co}{{\mathbb C}}

\newcommand{\bb}{{\cal B}}

\newcommand{\al}{{\alpha}}     
     
     \newcommand{\lm}{{\lambda}}
    
\newcommand{\om} {{\omega}}

\newcommand{\N}{{\mathbb N}}  
  
\def\x{\relax\ifmmode {\mbox{*}}\else*\fi}

\newcommand{\CB}{\mathcal{CB}}

\newcommand{\ed}{\end{document}}

\begin{document}
\pagestyle{myheadings} \markboth{\centerline{\small{\sc
            A.~Ya.~Helemskii}}}
         {\centerline{\small
         {\sc Projective tensor product of protoquantum spaces}}} 
 \title{\bf Projective tensor product of protoquantum spaces}
 \footnotetext{This research was supported by the Russian
Foundation for Basic Research (grant No. 15-01-08392).}

\author{\bf A.~Ya.~Helemskii}

\date{}

\maketitle

\begin{abstract}
A proto-quantum space is a (general) matricially normed space in the sense of Effros and Ruan 
presented in a `matrix-free' language. We show that these spaces have a special (projective) tensor 
product possessing the universal property with respect to completely bounded bilinear operators. We 
study some general properties of this tensor product (among them a kind of adjoint associativity), 
and compute it for some tensor factors, notably for $L_1$ spaces. In particular, we obtain what 
could be called the proto-quantum version of the Grothendieck theorem about classical projective 
tensor products by $L_1$ spaces. At the end, we compare the new tensor product with the known 
projective tensor product of operator spaces, and show that the standard construction of the latter 
is not fit for general proto-quantum spaces. 
\end{abstract}

\setcounter{section} {0} 
\section{Introduction}
In their paper~\cite{er1}, Effros and Ruan introduced and investigated the important notion of a 
matricially normed space. Very soon, after the discovery of Ruan Representation Theorem~\cite{ru2}, 
 the great majority of papers and monographs was dedicated only to the outstanding special class 
of these structures, the ${\cal L}^\ii$--matricially normed spaces. Now the latter are called 
abstract operator spaces (or just operator spaces), and sometimes quantum spaces. The theory of 
operator spaces is very rich and well-developed. It is presented in widely known  
textbooks~\cite{efr,paul,pis,blem}). 

 On the other hand, already in~\cite{er1,ru2} it was demonstrated that  matricially normed 
spaces are a subject of considerable interest even outside the class of operator spaces, that is 
without the assuming of the second axiom of Ruan. In this paper we return to general matricially 
normed spaces; however, presented in the equivalent `non-coordinate', or `matrix-free' language. 
(The latter seems to be more convenient for us in this circle of questions). We hope that our 
observations, apart from some results in the cited papers, also show that general matricially 
normed spaces deserve an independent interest. Moreover, in their study we sometimes come to things 
that look very different from what we know about operator spaces. 

Our main point is that the general matricially normed spaces, called in this paper proto-quantum 
spaces, have a tensor product, possessing the universal property relative to the class of 
completely bounded bilinear operators. In the context of operator spaces, such a tensor product was 
discussed in~\cite[II.7.1]{efr} and, in the non-coordinate language, in~\cite[Ch.7.2]{heb2}. 
Therefore, following the terminology of these textbooks, we call this tensor product projective. 
(Note that another kind of tensor product, the Haagerup tensor product, was already discussed 
in~\cite{er1,ru2}). 

The contents of the paper are as follows. 
 
 The second and third sections contain initial definitions, notably of a proto-quantum space, 
 of a completely 
 bounded operator and of a completely bounded bilinear operator. Also the simplest examples are
 presented, in particular, the maximal and minimal proto-quantization of a given normed space.
 Among several observations, we show that the maximal proto-quantization always gives rise to an
 ${\cal L}^1$--space, and that every bounded functional with the domain an ${\cal 
 L}^{p}$--space and with the range $\co$ that was made an ${\cal L}^q$--space, is `automatically' 
 completely  bounded, provided $p\le q$. (As to the notation ` ${\cal L}^q$ ', here and thereafter we 
 mean the non--coordinate counterpart of the notion of a matricially normed
 ${\cal L}^p$--space, initially introduced in~\cite{er1}).

In Section 4 we consider several further examples of proto-quantum not quantum spaces. In 
particular, we introduce what we call the standard proto-quantization of the space $L_p(X,E)$, 
where $E$ is a proto-quantum space. (Among them one can find the ${\cal L}^p$--space $L_p(X,E)$, 
where $E$ is an ${\cal L}^p$--space). Also we show that some bilinear operators, related to these 
spaces, are completely contractive; this will be used in subsequent sections. 

 In Section 5  we define the non-completed projective tensor product of proto-quantum spaces, 
 denoted by ` $\ot_{pop}$ ', and its `completed' version, denoted by ` $\widehat\ot_{pop}$ '. We prove 
 the respective existence theorems by displaying relevant explicit constructions. 
 
In Section 6 we present some examples of the computation of the introduced tensor product. It turns 
out that, just as in the case of the classical projective tensor product of normed spaces, the 
especially nice tensor factors are $L_1$--spaces. As the base of most applications, we show that 
 for all proto-quantum spaces $E$ and $F$ we have 
\[
L_1(X,E)\widehat\ot_{pop}L_1(Y,F)\simeq L_1(X\times Y,E\widehat\ot_{pop}F).
\]
(Here and thereafter ` $\simeq$ ' means a completely isometric isomorphism). Another frequently 
used fact is that, under some assumptions on  $E$ and, $F$ we have 
\[
E\widehat\ot_{pop}F\simeq E\widehat\ot_{pr}F,
\]
where on the right is the classical projective tensor product of our spaces, made a proto-quantum 
space according to a certain recipe in Section 4. Combining these two theorems, we obtain that for 
a $p$--convex proto-quantum space $E$ (in particular, an ${\cal L}^p$--space) and the complex 
plane, considered as an ${\cal L}^p$--space, we have 
\[
L_1(X,\co)\widehat\ot_{pop}E\simeq L_1(X,E).
\] 
This result can be considered as a version, for proto-quantum spaces, of the Grothendieck Theorem 
on tensoring by $L_1$-spaces (cf., e.g.,~\cite[\S2, n$^o$2]{gro}).

 At the beginning of the next chapter we extend to general proto-quantum 
spaces the method of the quantization of a given space of completely bounded operators, first 
suggested in papers~\cite[p.140]{ro},~\cite{blp},~\cite{er2}; see also the textbooks 
~\cite[I.3.2]{efr} or~\cite[Ch.8.7]{heb2}. Then we establish the suitable form of the so-called law 
of adjoint associativity, connecting spaces of operators with tensor products. (The form of that 
law in the context of the classical functional analysis is presented, e.g., 
in~\cite[Ch.6.1]{heb2}). Namely, for proto-quantum spaces $E,F,G$, the space 
$\mathcal{CB}(E\ot_{pop}F,G)$ is, in a natural way, (completely) isometrically isomorphic to 
$\mathcal{CB}(F,\mathcal{CB}(E,G))$ and to $\mathcal{CB}(E,\mathcal{CB}(F,G))$. 
 
 (We recall that the mentioned method essentially differs from the initial 
approach to what to call a dual matricially normed space. This approach was considered 
in~\cite{er1}), and, as it was shown there, has some advantages. However, its essential drawback is 
that it does not lead to the  adjoint associativity). 
 
In the concluding Section 7 we compare the introduced tensor product ` $\ot_{pop}$ ' with what 
could be called its prototype. By this we mean the well-known projective tensor product of operator 
spaces, denoted here by ` $\ot_{op}$ ', that was independently discovered by 
Blecher/Paulsen~\cite{blp} and Effros/Ruan~\cite{er2}. For operator spaces (i.e. when the second 
axiom of Ruan is fulfilled)  both tensor products coincide. However, for general proto-quantum 
spaces the standard formulae for the respective norms give different numbers: in the case of ` 
$\ot_{op}$ ' they are essentially greater than in the case of ` $\ot_{pop}$ '. As an example, we 
consider the projective tensor square of a certain proto-quantum space, and for every $n$ we 
display an element of its amplification, for which the first number is $n^2$, whereas the second is 
$n$. 

 \section{Proto-quantum spaces and their first examples}

As it was said, we use in this paper the so-called non-coordinate approach to the structures in 
question, and not the more widespread `matrix' approach, as in the 
textbooks~\cite{efr,paul,pis,blem}. 
Some of our terms and notation are contained in~\cite{heb2}, where practically only (abstract) 
operator spaces, called there quantum spaces, were 
 considered. For the convenience of the reader, we shall briefly repeat some of the most needed initial 
 definitions.

\bigskip
To begin with, we choose an arbitrary separable infinite-dimensional Hilbert space, denote it by 
$L$ and fix it throughout the whole paper. We write $\bb$ instead of $\bb(L)$, the Banach algebra 
of all bounded operators on $L$ with the operator norm, usually denoted just by $\|\cd\|$. The 
symbol $\ot$ is used for the (algebraic) tensor product of linear spaces and for 
 elementary tensors. The symbols $\ot_{pr}$ and $\ot_{in}$ denote the non-completed projective and 
 injective  tensor product of normed spaces, respectively.
The complex-conjugate space of a linear space $E$ is denoted by $E^{cc}$. The identity operator on 
a linear space $E$ is 
 denoted by $\id_E$, and we write $\id$ instead of $\id_L$. 

For $\xi,\eta\in L$ we denote by $\xi\circ\eta$ the rank 1 operator on $L$, taking $\zeta$ to 
$\la\zeta,\eta\ra\xi$ . Recall that $\|x\circ y\|=\|x\|\|y\|$. 

Denote by $\cK$ the (non-closed) two-sided ideal of $\bb$, consisting of finite rank bounded 
operators. Recall that there is a linear isomorphism $L\ot L^{cc}\to\cK$, well defined by taking 
$\xi\ot\eta$ to $\xi\circ\eta$. For $p\in[1,\ii]$ we denote by $\|\cd\|_p$ the norm  of the $p$-th 
Schatten class on $\cK$, and write $\cK_p:=(\cK,\|\cd\|_p$); in particular, $\cK_\ii$ is $\cK$ with 
the operator norm. 

\medskip
In what follows we need the triple notion of the so-called amplification. First, we amplify linear 
spaces, then linear operators and finally bilinear operators. 

The {\it amplification} of a given linear space $E$ 
 is the tensor product $\cK\ot E$. Usually we briefly denote it by $\cK E$, and an elementary
tensor $a\ot x; a\in\cK, x\in E$, by $ax$. Note that $\cK E$ is a bimodule over the algebra $\bb$ 
with the outer multiplications, denoted by ` $\cd$ ' and well defined by $a\cd(bx):=(ab)x$ and 
$(ax)\cd b:=(ab)x$. 
\begin{definition}
 A semi-norm on $\cK E$ is called {\it proto-quantum
semi-norm}, or briefly, {\it $PQ$--semi-norm} on $E$, if the $\bb$-bimodule $\cK E$ is contractive, 
that is we always have the estimate $\|a\cd u\cd b\|\le\|a\|\|u\|\|b\|$. A $PQ$--semi-norm on $E$ 
is called {\it quantum semi-norm}, or briefly, {\it $Q$--semi-norm} on $E$, if for  $u,v\in\cK E$ 
and (ortho)projections $P,Q\in\bb, PQ=0$ we always have $\|P\cd u\cd P+Q\cd v\cd Q\|=\max\{\|P\cd 
u\cd P\|,\|Q\cd v\cd Q\|\}$. 

The space $E$, endowed with a $PQ$--semi-norm, is called {\it semi-normed proto-quantum space}, or 
briefly, {\it semi-normed $PQ$--space}. In the case of a normed $PQ$--space we usually omit the 
word `normed'. 

In a similar way we use the terms {\it semi-normed $Q$--space} and (just) {\it $Q$--space}.
\end{definition}
\begin{remark}
By their definition, $Q$--spaces can be treated as a special case of the so-called Ruan bimodules, 
considered in~\cite{her} and~\cite{wit}. 
\end{remark}
\begin{remark}
Let us recall, for the convenience of the reader, the way of the translation from the `matrix 
language' to the `non-coordinate language'. Let $E$ be a matricially normed space in the sense 
of~\cite{er1}, and we are given $u\in\cK E$. Clearly, there exists a finite rank projection $P$ 
such that $u$ has the form $\sum_{k+1}^na_kx_k; a_k=P\cd a_k\cd P, x_k\in E$. We choose an 
arbitrary orthonormal basis in $P(L)$ and denote by $(a^k_{ij})$ the matrix, in this basis, of the 
restriction $a_k$ to $P(L)$. Then we take the matrix $(u_{ij}:=\sum_ka^k_{ij}x_k)$ with entries in 
$E$ and set $\|u\|:=\|(u_{ij})\|$. It is easy to show that $\|u\|$ does not depend  on the choice 
of $P$ and of a basis in $P(L)$, and that the function $u\mt\|u\|$ is a $PQ$--norm on $E$. 
\end{remark}

 A semi-normed $PQ$--space $E$ becomes a semi-normed space (in the usual sense), if for $x\in E$ 
we set $\|x\|:=\|Qx\|$, where $Q$ is an arbitrary rank 1 operator of norm 1. Obviously, the 
resulting semi-norm does not depend on the particular choice of $Q$. The obtained semi-normed space 
is called {\it underlying space} of a given $PQ$--space, and the latter is called a {\it 
proto-quantization} (briefly $P$--quantization) or, if we deal with a $Q$--space, a {\it 
quantization} of the former. (The term `quantization' ascends to the seminal Effros' 
lecture~\cite{ef5}. Indeed, in the space $E=\co\ot E$ commutative scalars from $\co$ are replaced 
by the `non-commutative scalars' from $\cK$; a typical device of `quantum mathematics'). 
\begin{proposition}
Let $E$ be a semi-normed $PQ$--space with a normed underlying space. Then the semi-norm on $\cK E$ 
    is a norm.
\end{proposition}
 The proof, given in~\cite[Prop. 1.2.2]{heb2} for $Q$--spaces is valid, without any modification, 
 for $PQ$--spaces too.

%
%

\begin{example} \label{ex1}
 Every non-zero normed space, say $E$, has a lot of $P$--quantizations. Among them we distinguish 
 the so-called
{\it maximal} and {\it minimal}, denoted by $E_{\max}$ and $E_{\min}$,  respectively. The first 
space is obtained by the endowing $\cK E$ with the norm of $L\ot_{pr}L^{cc}\ot_{pr} E=\cK_1\ot_{pr} 
E$, and the second with the norm of $L\ot_{in}L^{cc}\ot_{in}E=\cK_\ii\ot_{in}E$. (Evidently, the 
first of these $PQ$--norms is never a $Q$--norm  whereas it is not difficult to show that the 
second one is always a $Q$--norm.) 

As a matter of fact, the first norm is maximal in the sense that it is the greatest of all norms of 
$P$--quantizations of $E$. Indeed, we easily see that the norm on $L\ot L^{op}\ot E$, corresponding 
to any given $PQ$--norm on $E$, is a cross-norm. But among all cross-norms there is a greatest one, 
and it is exactly the norm on $L\ot_{pr}L^{op}\ot_{pr}E$. In a similar sense the second norm is 
minimal, but this statement will be justified a little bit later. 

As to $E_{\max}$, it is a member of the whole family of $P$--quanizations of $E$, denoted by 
${^{(p)}}E; 1\le p\le\ii$; they are obtained by endowing $\cK E$ with the norm of 
${\cK_p}{\ot_{pr}}E$. Clearly, we have ${^{(1)}}E=E_{\max}$. In particular, among various 
$P$--quantizations of $\co$ we distinguish $PG$--spaces ${^{(p)}}\co$; we see that the 
amplification of such a space is identified with $\cK_p$. 
Moreover, there is only one $P$--quantization of $\co$ which is a quantization, and this is 
${^{(\ii)}}\co$. 
\end{example}

\medskip
In what follows, if numbers $\lm_k\ge0; k=1,...,n$ are given, we shall understand the expression 
 $(\sum_{k=1}^n\lm_k^p)^{\frac{1}{p}}$ as $\max\{\lm_1,...,\lm_n\}$ in the case $p=\ii$.

\medskip
We shall say that a projection $P\in\bb$ is a support of an element $u\in\cK E$, if we have $P\cd 
u\cd P=u$. 

For $p\in[1,\ii]$ we shall say that a $PQ$--space $E$ is an {\it ${\cal L}^p$--space}, 
respectively, {\it $p$--convex space} and {\it $p$--concave space}, if for every 
$u_1,\dots,u_n\in\cK E$ with pairwise orthogonal supports
we have $\|\sum_{k=1}^nu_k\|=(\sum_{k=1}^n\|u_k\|^p)^{\frac{1}{p}}$,   
respectively, $\|\sum_{k=1}^nu_k\|\le(\sum_{k=1}^n\|u_k\|^p)^{\frac{1}{p}}$, 
$\|\sum_{k=1}^nu_k\|\ge(\sum_{k=1}^n\|u_k\|^p)^{\frac{1}{p}}$. Obviously, it is sufficient to have 
similar relations for the case $n=2$. We see that ${\cal L}^\ii$--space is just another name for a 
$Q$--space. Clearly, every $PQ$--space is 1--convex and $\ii$--concave. Moreover, ${^{(p)}}\co$ is 
evidently an  ${\cal L}^p$--space. 

Throughout the paper, for $a,b\in\cK$ we shall write $a\approx b$ provided we have $SaT=b$ for some 
unitary operators $S,T\in\bb$. Similarly, for $u,v\in\cK E$, where $E$ is a $PQ$-- space, we shall 
write $u\approx v$ provided $S\cd u\cd T=v$ for $S,T$ as above. Clearly, $a\approx b$ implies 
$\|a\|_p=\|b\|_p$ for all $p\in[1,\ii]$, and $u\approx v$ implies $\|u\|=\|v\|$. It is well known 
(and easy to show) that for every $a\in\cK$ we have

\begin{align}
a\approx h, \q {\rm where} \q h=\sum_{k=1}^ns_kP_k
\end{align}
 for some pairwise orthogonal rank one projections $P_k\in\cK$, and $s_k\ge0$.

\begin{proposition}
 If $E$ is a $p$--convex $PQ$--space, $p$--concave $PG$--space or an ${\cal L}^p$--space, then, for all 
 $a\in \cK, x\in E$, we have $\|ax\|\le\|a\|_p\|x\|$, $\|ax\|\ge\|a\|_p\|x\|$ or $\|ax\|=
 \|a\|_p\|x\|$, respectively. In  particular, for every $E$ we have $\|ax\|\le\|a\|_1\|x\|$.
\end{proposition}
 \begin{proof}
Let $h$ be as in (2.1). Then we have $ax\approx hx$. But  $hx=\sum_{k=1}^ns_kP_kx$, where the  
summands have pairwise orthogonal supports, namely $P_k$. 
  Therefore in the `convex' case we have $\|ax\|=\|hx\|\le
 (\sum_{k=1}^n\|s_kP_kx\|^p)^{\frac{1}{p}}=(\sum_{k=1}^ns_k^p)^{\frac{1}{p}}\|x\|$, where, as we recall,
 $(\sum_{k=1}^ns_k^p)^{\frac{1}{p}}$ is just $\|h\|_p$, that is $\|a\|_p$. 
 The remaining cases are treated in a similar way.
 \end{proof}

%

\begin{example}
We shall show that the maximal $P$--quantization of a given normed space $E$ is an ${\cal 
L}^1$--space. (And thus every normed space can be made an ${\cal L}^1$--space). 

Indeed, consider orthogonal projections $P,Q\in\bb$ and the subspaces $\cK_1^P:=\{PaP; a\in\cK\}$, 
$\cK_1^Q:=\{QaQ; a\in\cK\}$ and $\cK_1^{P,Q}=\{PaP+QaQ;a\in\cK\}$ in $\cK_1$. Clearly, we have 
$\cK_1^{P,Q}=\cK_1^P\oplus_1\cK_1^Q\in\cK_1$, where ` $\oplus_1$ ' is a sign of the $\ell_1$--sum 
of normed spaces. 

It is well known (and easy to check) that the operator $j:\cK_1\to \cK_1^{P,Q}:a\mt Pap+QaQ$ is 
contractive (in fact, it is a norm 1 projection). Consider the operators 
$j\ot_{pr}\id_E:\cK_1\ot_{pr}E\to\cK_1^{P,Q}\ot_{pr}E$, and also 
$i\ot_{pr}\id_E:\cK_1^{P,Q}\ot_{pr}E\to\cK_1\ot_{pr}E$, where $i$ is the respective natural 
embedding. Both of them, being projective tensor product of contractive operators, are contractive 
themselves. But their composition is evidently the identity operator on $\cK_1^{P,Q}\ot_{pr}E$. It 
follows that $i\ot_{pr}\id_E$ is an isometry (whereas $j\ot_{pr}\id_E$ is a strict coisometry). 

Now suppose that $u,v\in\cK_1\ot_{pr}E$ have $P$ and $Q$ as their respective supports. Observe that 
for every $w\in\cK_1\ot_{pr}E$ the equality $w=P\cd w\cd P+Q\cd w\cd Q$ means exactly that 
$w\in\cK_1^{P,Q}$. Consequently, elements $u,v$ and $u+v$ have the same norms in 
$\cK_1^{P,Q}\ot_{pr}E$ as in $\cK_1\ot_{pr}E$. 
 
Finally, recall the known connection between the operations ` $\oplus_1$ ' and \\ ` $\ot_{pr}$ '. 
In our situation we have the isometric isomorphism 
$I:\cK_1^{P,Q}\ot_{pr}E\to(\cK_1^{P}\ot_{pr}E)\oplus_1(\cK_1^{Q}\ot_{pr}E)$, well defined by taking 
$(a+b)\ot x$ to $a\ot x+b\ot x$. But for $u,v$ as of elements of $\cK_1^{P,Q}\ot_{pr}E$ we see that 
$I(u)\in\cK_1^{P}\ot_{pr}E$ and $I(v)\in\cK_1^{Q}\ot_{pr}E$. Therefore 
\[
\|u+v\|=\|I(u+v)\|=\|I(u)+I(v)\|=\|I(u)\|+\|I(v)\|=\|u\|+\|v\|,
\]
 and we are done. 
\end{example}

Note, however, that the $PG$--space ${^{(p)}}E$ for $E\ne\co$ and $p>1$ is, generally speaking, not 
an ${\cal L}^p$--space. 

\begin{remark}
Recall that numerous examples, all of them concerning $Q$--spaces, are presented in the cited 
textbooks. They include the example, which is the most important in the whole theory of 
$Q$--spaces, being, in a sense, universal~\cite{ru2}~\cite{er5}. This is the so-called concrete 
quantization of a space, consisting of operators. But we do not need this material in the present 
paper. 
\end{remark}

\section{Completely bounded linear and bilinear operators}

Suppose that we are given an operator $\va:E\to F$
between linear spaces. The {\it amplification of $\va$} is the operator $\va_\ii: 
\cK E\to\cK F$, well defined on elementary tensors by $ax\mt a\va(x)$. Clearly, $\va_\ii$ is a 
morphism of $\bb$-bimodules. 

\begin{definition}
 An operator $\va$, connecting  semi-normed $PQ$--spaces, is called {\it completely bounded}, respectively, 
 {\it completely contractive}, if its amplification is bounded, respectively, contractive in the 
 usual sense. We set $\|\va\|_{cb}:=\|\va_\ii\|$.
\end{definition}
In a similar way we define the notions of a {\it completely isometric operator} and of a {\it 
completely isometric isomorphism.} 

 \medskip
If $\va$ is bounded, being considered in the context of the respective underlying semi-normed 
spaces, we say that it is (just) {\it bounded} and denote its respective operator semi-norm, as 
usual, by $\|\va\|$. Every completely bounded linear operator is obviously bounded, and we have 
$\|\va\|\le\|\va\|_{cb}$. 

Denote by $\CB(E,F)$ the subspace in $\bb(E,F)$, consisting of completely bounded linear operators. 
It is a normed space with respect to the norm $\|\cd\|_{cb}$. 

\medskip
Some linear operators between $PQ$--spaces that are bounded, 
are `automatically' completely bounded. Here is an observation of that kind. 

\begin{proposition}
 Let $E$ be an ${\cal L}^p$--space or, more general, $p$--concave $PQ$--space for some $p\in[1,\ii]$ 
Then every bounded functional $f:E\to{^{(q)}}\co$, where $q\ge p$, is completely bounded, and 
$\|f\|_{cb}:=\|f\|$. 
\end{proposition}

\begin{proof}
 Take an arbitrary $u\in\cK E$. 
  Setting in (2.1) $a:=f_\ii(u)$ and $v:=S\cd u\cd T$, where $S,T$ are relevant unitary operators, 
 we see that for some pairwise orthogonal rank 1 projections $P_k$ and $s_k\ge0$ we have

\begin{align}
f_\ii(v)=\sum_ks_kP_k, \q \|u\|=\|v\|\q {\rm and} \q \|f_\ii(u)\|_q=
\|f_\ii(v)\|_q=(\sum_{k=1}^ns_k^q)^{\frac{1}{q}}.
\end{align}
Therefore it suffices to prove that 
$(\sum_{k=1}^ns_k^q)^{\frac{1}{q}}\le\|f\|\|v\|$.

 Denote by $\zeta$ the prime $n$-th root of 1 and set, 
for $m=1,\dots,n$, $W_m:=\sum_{k=1}^n\zeta^{mk}P_k$ and $W'_m:=\sum_{k=1}^n\zeta^{-mk}P_k$. Then a 
routine calculation shows that for all $a\in\bb$ we have 
$\sum_{m=1}^nW'_maW_m=n(\sum_{k=1}^nP_kaP_k)$. Hence, representing $v$ as a sum of elementary 
tensors, we see that $\sum_{m=1}^nW'_m\cd v\cd W_m=n(\sum_{k=1}^nP_k\cd v\cd P_k)$. From this we 
have 
\begin{align}
\|\sum_{k=1}^nP_k\cd v\cd P_k\|\le\frac{1}{n}\sum_{m=1}^n\|W'_m\cd v\cd W_m\|\le
\frac{1}{n}\sum_{m=1}^n\|W'_m\|\|v\|\|W_m\|\le\|v\|.
\end{align}

Since $P_kbP_k$ is proportional to $P_k$ for all $b\in\bb$ and $k=1,\dots,n$, we easily see that 
for these $k$ we have $P_k\cd v\cd P_k=P_kx_k $ for some $x_k\in E$. 

Therefore, by (3.1), we have for all $k$ that
\[
s_kP_k=P_kf_\ii(v)P_k=f_\ii(P_k\cd v\cd P_k)=f_\ii(P_kx_k)=f(x_k)P_k,
\]
hence $f(x_k)=s_k$, and consequently
\[
s_k\le\|f\|\|x_k\|=\|f\|\|P_kx_k\|=\|f\|\|P_k\cd v\cd P_k\|.
\]
Therefore, taking into account that $E$ is $p$--concave, we have
 \[
\left(\sum_{k=1}^ns_k^q\right)^{\frac{1}{q}}\le
\|f\|\left(\sum_{k=1}^n\|P_k\cd v\cd P_k\|^q\right)^{\frac{1}{q}}\le
\]
\[ 
\|f\|\left(\sum_{k=1}^n\|P_k\cd v\cd P_k\|^p\right)^{\frac{1}{p}}\le\|f\|\|\sum_{k=1}^nP_k\cd v\cd P_k\|.
\]
It remains to apply (3.2). 
\end{proof}

 In particular (cf.~\cite{er1}), {\it for every $PQ$--space $E$ every bounded functional 
$f:E\to{^{(\ii)}}\co$  is completely bounded, and $\|f\|_{cb}:=\|f\|$.} 

Note that the latter assertion immediately implies that for a normed space $E$ the norm on $\cK E$, 
given by $\|u\|:=\sup\{\|f_\ii(u)\|_\ii;f\in E^*; \|f\|\le1\}$, is the smallest among all norms of 
$P$--quantizations of $E$. But this is exactly the norm on $L\ot_{in}L^{cc}\ot_{in}E$. This 
justifies the word `minimal' for the latter norm (see above). Also we see that we got a $Q$--norm.

 On the other hand, contrary to the situation with the space of all bounded operators,
the space ${\cal CB}(E,F)$ can be very scanty. 
The following observation is taken from~\cite{ru2}. 
 
\begin{proposition}
 Let $E$ be $p$--convex, $F$ be $q$--concave $PQ$--spaces, and $p>q$. Then there 
is no non-zero completely bounded operators from $E$ into $F$. 
\end{proposition} 

\begin{proof}
Let $\va:E\to F$ be an arbitrary non-zero operator; our task is to show that it is not completely 
bounded. Take $x\in E$ with $\va(x)\ne0$. Since $p>q$, for every $n\in\N$ there exist $m\in\N$ such 
that $m^{\frac{1}{q}}>(n\|x\|/\|\va(x)\|)m^{\frac{1}{p}}$. Take pairwise orthogonal rank 1 
projections $P_k; k=1,...,m$ and set $u:=\sum_kP_kx\in\cK E$; then 
$\va_\ii(u)=\sum_kP_k\va(x)\in\cK F$. We see that elements $P_kx\in\cK E$ as well as 
$P_k\va(x)\in\cK F$ have pairwise orthogonal supports. Therefore we have  

\[
\|\va_\ii(u)\|\ge\left(\sum_k\|P_k\va(x)\|^q\right)^{\frac{1}{q}}=
\left(m\|\va(x)\|^q\right)^{\frac{1}{q}}=\|\va(x)\|m^{\frac{1}{q}}>
\]

\[
\|\va(x)\|\frac{n\|x\|}{\|\va(x)\|}m^{\frac{1}{p}}=
n\|x\|m^{\frac{1}{p}}\ge n\left\|\sum_kP_kx\right\|=n\|u\|.
\]
Since $n$ is arbitrary, this means that the operator $\va_\ii$ is not bounded. 
\end{proof}

\medskip
However, most of various known counter-examples (one of the earliest is due to Tomiyama~\cite{tom}) 
concern $Q$--spaces; see the textbooks cited above. 

\bigskip
To amplify bilinear operators (in what follows, we shall say, for brevity, `bioperators'), we shall 
use a certain operation that imitate tensor product of operators on our Hilbert space $L$ but does 
not lead out of $L$. (Within the `matrix' approach we would have to use the Kronecker product of 
matrices). 

In what follows, the symbol $\hil$ is used for the Hilbert tensor product 
 of Hilbert spaces, as well as of bounded operators, acting on these spaces. 
By virtue of Riesz/Fisher Theorem, we can arbitrarily choose a unitary isomorphism $\iota:L\hil 
L\to L$ and fix it throughout the whole paper. Following~\cite{he6}, for $\xi,\eta\in L$ we denote 
the vector $\iota(\xi\ot\eta)\in L$ by $\xi\di\eta$, and for $a,b\in\bb$ we denote the operator 
$\iota(a\hil b)\iota^{-1}$ on $L$ by $a\di b$ ; obviously, the latter is well defined by the 
equality $(a\di b)(\xi\di\eta)= a(\xi)\di b(\eta)$. Evidently, we have the identities 
\begin{align} \label{21}
(a\di b)(c\di d)=ac\di bd, \qq \|\xi\di\eta\|=\|\xi\|\|\eta\| \qq{\rm and}\qq \|a\di b\|=\|a\|\|b\|.
\end{align}
%
%
Now suppose that we are given a bioperator $\rr:E\times F\to G$ between linear spaces. Its {\it 
amplification} is the bioperator $\rr_\ii:\cK E\times\cK F\to\cK G$, well defined on elementary 
tensors by $\rr_\ii(ax,by)=(a\di b)\rr(x,y)$. 
 
\begin{remark}
We do not consider here another, different version of the amplification of a bioperator, that would 
lead us to 
the important notion of the Haagerup tensor product of $PQ$--spaces (cf.~\cite{er1} and, in the 
context of $Q$--spaces,~\cite{ble}~\cite{blp} and also the textbooks~\cite{efr}~\cite{heb2}). 
\end{remark}

\begin{definition}
 A bioperator $\rr$, connecting  semi-normed $PQ$--spaces, is called {\it completely bounded}, 
 respectively, 
 {\it completely contractive} if its amplification is bounded, respectively, contractive in the 
 usual sense. We set $\|\rr\|_{cb}:=\|\rr_\ii\|$. 
\end{definition}

Note that, as it is easy to see, after restricting ourselves to $Q$--spaces and translating this 
definition back to the `matrix language', we shall obtain the standard definition of completely 
bounded (and completely contractive) bioperator between operator spaces (see~\cite[p.126]{efr}). 

Here is another example of the `automatic complete boundedness'. {\it If $E, F$ are $PQ$--spaces, 
and $f:E\to\co$, $g:F\to\co$ are bounded functionals, then the bilinear functional $f\times 
g:E\times F\to{^{(\ii)}}\co:(x,y)\mt f(x)g(y)$ is completely bounded, and $\|f\times 
g\|_{cb}=\|f\times g\|=\|f\|\|g\|$.}  This can be easily deduced from Proposition 3.2 with the help 
of the formula $(f\times g)_\ii(u,v)=f_\ii(u)\di g_\ii(v), u\in\cK E, v\in\cK F$. 

As a good exercise, we can mention the situation with the inner product bilinear functional 
  $H\times H^{cc}\to{^{(\ii)}}\co:(x,y)\mt\la x,y\ra$, where $H$ is a Hilbert space. It is completely 
  contractive, if we endow both $H$ and $H^{cc}$ with the maximal $PQ$--norm (cf. Example \ref{ex1}),
and it is not completely bounded, if we endow them with the minimal $Q$--norm. 

\section{Further examples of proto-quantum spaces and related bilinear operators}
We introduce here several examples of $PQ$--spaces. Later some of them will show especially good 
behavior as tensor factors. 
\begin{example} 
Let $(X,\mu)$ be a measure space and $F$ be an arbitrary $PQ$--space. We want to endow the normed 
space $L_p(X,F);1\le p\le\ii$ of relevant $F$-valued measurable functions on $X$ with a $PQ$--norm. 

\medskip
As a preliminary step, consider the (non-completed) normed space $L_p(X,\cK F)$ 
 and note that it is a $\bb$-bimodule with the outer multiplications defined by
$$
[a\cd\bar x](t):=a\cd[\bar x(t)]\q{\rm and}\q [\bar x\cd b](t):=[\bar x(t)]\cd
b; \q a,b\in\bb,\bar x\in L_p(X,\cK F), t\in X.
$$
A routine calculation shows that this bimodule is contractive. 

Now consider the operator $\al:\cK(L_p(X,F))\to L_p(X,\cK F)$, well defined on elementary tensors 
by taking $ax$ to the $\cK F$-valued function $\bar x(t):=a(x(t))$. Introduce  the semi-norm on 
$\cK(L_p(X,F))$ by setting $\|u\|:=\|\al(u)\|$. Observe that $\al$ is a $\bb$-bimodule morphism: to 
show this, it is sufficient to consider respective elementary tensors. 

 Thus, there is an isometric morphism of the semi-normed bimodule $\cK(L_p(X,F))$ into a contractive 
 $\bb$-bimodule. It follows immediately that the former bimodule is itself contractive, hence the 
introduced semi-norm on $\cK(L_p(X,F))$ is a $PQ$--semi-norm on $L_p(X,F)$. Further, for an 
arbitrary rank 1 operator $Q\in\cK; \|Q\|=1$ and $x\in L_p(X,F)$ we have $\|Q[x(t)]\|=\|x(t)\|$ for 
all $t\in X$. Therefore for $Qx\in\cK(L_p(X,F))$ we easily have $\|Qx\|=\|x\|$. 
This means that the underlying semi-normed space of the constructed $PQ$--space is
the `classical'  $L_p(X,F)$. Consequently, Proposition 2.4 guarantees that the 
$PQ$--semi-norm on $L_p(X,F)$ is actually a norm. 
\end{example}

It is easy to verify that the $PQ$--space $L_p(X,F)$ is $p$--convex or $p$--concave provided $F$ 
has the same property. In particular, if $F$ is an ${\cal L}_p$--space, then $L_p(X,F)$ is also an 
${\cal L}_p$--space. 
Note also that the $PQ$--space $L_p(X,F)$ is not a $Q$--space whenever $p<\ii$, and $X$ is not a 
single atom. 
%
\begin{example} \label {ex8}
Now we want to introduce a $P$--quantization of the `classical' tensor product $E\ot_{pr}F$ of 
normed spaces, when one of tensor factors, say, to be definite, $F$, is a $PQ$--space. 

Consider the linear isomorphism $\beta: \cK(E\ot F)\to E\ot_{pr}(\cK F)$, well defined by taking 
$a(x\ot y)$ to $x\ot ay$, and introduce a norm on $\cK(E\ot F)$ by setting $\|U\|:=\|\beta(U)\|$. 
The space $E\ot_{pr}(\cK F)$, as a projective tensor product of a normed space and a contractive 
$\bb$-bimodule, has itself a standard structure of a contractive $\bb$-bimodule. The same, because 
$\beta$ is obviously a $\bb$-bimodule morphism, is true with $\cK(E\ot F)$. Thus $E\ot F$ becomes a 
$PQ$--space, and we must show that its underlying normed space is exactly $E\ot_{pr}F$. 

Denote the norm on $E\ot_{pr} F$ and on $E\ot_{pr}(\cK F)$ by $\|\cd\|_{pr}$, and the introduced 
$PQ$--norm, as well as the norm of the respective underlying space, just by $\|\cd\|$. 

Take an arbitrary $u\in E\ot F$. It is easy to check that the norm on the underlying space in 
question is a cross-norm, hence $\|u\|\le\|u\|_{pr}$. Therefore our task is to show that for 
 a rank 1 projection $P\in\cK$ we have $\|Pu\|\ge\|u\|_{pr}$.  
 
Take an arbitrary representation of $\beta(Pu)$ as $\sum_{k=1}^nx_k\ot w_k; x_k\in E,w_k\in\cK F$. 
Obviously, $P\cd w_k\cd P=P y_k$ for some $y_k\in F; k=1,...,n$. Therefore 
$\sum_{k=1}^n\|x_k\|\|w_k\|\ge \sum_{k=1}^n\|x_k\|\|P\cd w_k\cd P\|=\sum_{k=1}^n\|x_k\|\|y_k\|$. 

But we have $\beta(Pu)=P\cd\beta(Pu)\cd P=\sum_{k=1}^nx_k\ot P\cd w_k\cd 
P=\beta(P[\sum_{k=1}^nx_k\ot y_k])$. It follows that $u=\sum_{k=1}^nx_k\ot y_k$ and consequently, 
$\sum_{k=1}^n\|x_k\|\|w_k\|\ge\|u\|_{pr}$. From this, by the definition of the norm on 
$E\ot_{pr}(\cK F)$, we have $\|Pu\|=\|\beta(Pu)\|\ge\|u\|_{pr}$. 
%
\end{example}

The introduced $PQ$--spaces participate in  some bioperators that we shall essentially use. Their 
study needs a certain extended version of the operation ` $\di$ '. Namely, if $E$ is a linear 
space, $a\in\cK$ and $u\in\cK E$, then we introduce in $\cK E$ the elements, denoted by  $a\di u$ 
and $u\di a$. They are well defined, if we 
set $a\di(\sum_kb_kx_k):=\sum_k(a\di b_k)x$ and $(\sum_kb_kx_k)\di a:=\sum_k(b_k\di a)x_k$.
We shall use the following properties of such an operation that may have an independent interest. 

As a preparatory step, for a given $e\in L;\|e\|=1$ we introduce  the operator $S$ on $L$, acting 
as $\zeta\mt e\di\zeta$; it is, of course, an isometry. It is easy to verify that for all $b\in\cK$ 
and $P:=e\circ e$ we have 
\begin{align}
b=S^*(P\di b)S \q {\rm and} \q P\di b=SbS^*. 
\end{align}

\begin{proposition} \label{pr7}
 Let $E$ be a $PQ$--space, $u\in\cK E$. Then

(i) for every $a\in\cK$ we have $\|a\di u\|=\|u\di a\|$ 

(ii) for every $Q\in\cK$ of rank 1 we have $\|Q\di u\|=\|Q\|\|u\|$.

(iii) for an arbitrary $a\in\cK$ we have $\|a\di u\|\le\|a\|_p\|u\|$ provided  $E$ is $p$--convex, 
$\|a\di u\|\ge\|a\|_p\|u\|$ provided  $E$ is $p$--concave and, as a corollary, $\|a\di u\| 
=\|a\|_p\|u\|$ provided $E\in{\cal L}_p$. In particular, for all $E$ we have $\|a\di 
u\|\le\|a\|_1\|u\|$ 
\end{proposition}
\begin{proof}
(i) Consider the unitary operator $\bigtriangleup$ on $L$, well defined by taking 
$\xi\di\eta;\xi,\eta\in L$ to $\eta\di\xi$. Obviously, for every $a,b\in\cK$ we have $b\di 
a=\bigtriangleup(a\di b)\bigtriangleup$. From this we easily deduce that for every $a\in\cK $ we 
have $\bigtriangleup\cd(a\di u)\cd\bigtriangleup=u\di a$. It remains to recall that the 
$\bb$-bimodule $\cK E$ is contractive. 

\medskip
(ii) We can assume that $\|Q\|=1$. Then $Q=\xi\circ\eta$ for some $\xi,\eta\in L; 
\|\xi\|=\|\eta\|=1$. Then, for $e$ and $P$ as above,
the formulae (4.1), being combined with the equalities $Q=R_1PR_2$ and $P=R_1^*QR_2^*$, where 
$R_1:=\xi\circ e$ and $R_2:=e\circ\eta$, imply that 
$$
Q\di b=(R_1\di\id)SbS^*(R_2\di\id)\q {\rm and}\q b=S^*(R_1^*\di\id)(q\di b)(R_2^*\di\id)S.
 $$
 Therefore, representing $u$ as a sum of elementary tensors, we obtain that
$$
Q\di u=[(R_1\di\id)S]\cd u\cd[S^*(R_2\di\id)] \q\q {\rm and} \q\q
u=[S^*(R_1^*\di\id)]\cd(q\di u)\cd[(R_2^*\di\id)S].
$$
But  all operators, participating in these equalities, have norm 1, and the bimodule $\cK E$ is 
contractive. Consequently, we have the estimate $\|Q\di u\|\le\|u\|$ and its inverse.

\medskip
(iii) By (2.1), for our $a$ there exist $h, P_k$ and $s_k$ with mentioned properties. If $S,T$ are 
relevant operators, we have $a\di u=(S\di\id)\cd(h\di u)\cd(T\di\id)$, hence $\|a\di u\|=\|h\di 
u\|=\|\sum_ks_kP_k\di u\|$. Further, the elements $P_k\di u$ have pairwise orthogonal supports, 
namely $P_k\di\id$. Combining this with (ii) and remembering, what is $\|h\|_p$, we have, in 
`convex' case, that $\|a\di u\|\le (\sum_{k=1}^n(s_k\|u\|)^p)^{\frac{1}{p}}=\|a\|_p\|u\|$. 
Similarly, in the `concave' case, we obtain the inverse estimate. 
\end{proof}

 Here are several applications. 
In the following proposition $p\in[1,\ii]$, and we consider $L_p(X,F)$, where $F$ is a given 
$PQ$--space, and also $L_p(X,{^{(p)}}\co)$ as $PQ$--spaces according to  Example 4.1.

\begin{proposition}
Let $F$ be $p$--convex. Then the bioperator $\rr:L_p(X,{^{(p)}}\co)\times F\to L_p(X,F)$, taking a 
pair  $(z,y)$ to the $F$-valued function $t\mt z(t)y; t\in X$, is completely contractive. 
\end{proposition}
\begin{proof}
 Recall the isometric operator $\al:\cK(L_p(X,F))\to L_p(X,\cK F)$ and
distinguish its particular case $\al_0:\cK(L_p(X),{^{(p)}}\co)\to L_p(X,\cK_p)$. Also introduce the 
bioperator ${\cal S}:L_p(X,\cK_p)\times\cK F\to L_p(X,\cK F)$, taking a pair $(\om,v)$ to the $\cK 
F$-valued function $t\mt\om(t)\di v; t\in X$. Consider the diagram 
\[
 \xymatrix @C+20pt{\cK(L_p(X,{^{(p)}}\co))\times\cK F \ar[r]^{\rr_\ii}\ar[d]_{\al_0\times\id_{\cK F}} 
& \cK(L_p(X,F)) \ar[d]^{\al} \\
L_p(X,\cK_p)\times\cK F \ar[r]^{{\cal S}} & L_p(X,\cK F) }, 
\]
 \noindent It is commutative: this is easy to check on elementary tensors in the respective 
amplifications. Therefore, for $w\in\cK(L_p(X,{^{(p)}}\co))$ and $v\in\cK E$ we have 
\begin{align}
\|\rr_\ii(w,v)\|=\|\al(\rr_\ii(w,v))\|=\|{\cal S}(\al_0(w),v)\|.
\end{align}
But it follows from Proposition 4.3(iii) that for all $\om\in L_p(X,\cK_p), v\in\cK F$ we have
\[
\|{\cal S}(\om,v)\|=\left(\int_X(\|\om(t)\di v\|^pdt\right)^{\frac{1}{p}}\le
\left(\int_X(\|\om(t)\|_p\|v\|)^pdt\right)^{\frac{1}{p}}=\|\om\|\|v\|.
\]
Setting in (4.2) $\om:=\al_0(w)$ and remembering that $\al_0$ is an isometry, we obtain that 
$\|\rr_\ii(w,v)\|\le\|w\|\|v\|$. 
\end{proof}

In the following proposition $E$ is a normed space, ${^{(p)}}E$ is its $P$--quantization from 
Example \ref{ex1}, $F$ and $E\ot_{pr}F$ are $PQ$--spaces from Example 4.2. 

\begin{proposition}  \label{pr}
Let $F$ be $p$--convex. Then the canonical bioperator $\vartheta:{^{(p)}}E\times F\to 
E\ot_{pr}F:(x,y)\mt x\ot y$, is completely contractive. In particular, the bioperator  
$\rr:{^{(p)}}\co\times F\to F:(\lm,x)\mt\lm x$ is completely contractive. 
\end{proposition}

 \begin{proof}
Consider the trilinear operator ${\cal T}:E\times \cK\times \cK F\to E\ot(\cK F):(x,a,v)\mt 
x\ot(a\di v)$. It gives rise to the bioperator  ${\cal S}:(E\ot\cK)\times\cK F\to E\ot\cK F:(x\ot 
a,v)\mt x\ot(a\di v)$. Being considered with the domain $E\times \cK_p\times \cK F$ and the range 
$E\ot_{pr}(\cK F)$, ${\cal T}$ is contractive by virtue of Proposition 4.3(iii); therefore ${\cal 
S}$ is contractive, taken with the domain $(E\ot_p\cK)\times\cK F$ and the same range. Now recall  
the isometric operator $\beta:\cK(E\ot_{pr}F)\to E\ot_{pr}\cK F$ and distinguish its particular 
case, the ``flip'' $\beta_0:\cK ({^{(p)}}E)\to E\ot_{pr}\cK_p$. Consider the diagram 
\[
\xymatrix@C+20pt{\cK({^{(p)}}E)\times\cK F \ar[r]^{\vartheta_\ii}\ar[d]_{\beta_0\times\id_{\cK F}}
& \cK(E\ot_{pr} F) \ar[d]^{\beta} \\
(E\ot_{pr}\cK_p)\times\cK F \ar[r]^{{\cal S}} & E\ot_{pr}\cK F },
\]
\noindent which is obviously commutative. Therefore for $w\in \cK({^{(p)}}E)$ and $v\in\cK F$ we 
have  
%
\[
\|\vartheta_\ii(w,v)\|=\|\beta(\vartheta_\ii(w,v))\|=\|{\cal S}(\beta_0(w),v\|
\le\|\beta_0(w)\|\|v\|=\|w\|\|v\|. 
\]
\end{proof}

Our third example of a completely contractive bioperator needs some preparatory observation which 
must be well known in its equivalent version for the `genuine' Hilbert tensor product of operators. 

\begin{proposition}
For $a,b\in\cK_p$ we have $\|a\di b\|_p=\|a\|_p\|b\|_p$. 
\end{proposition} 

\begin{proof} Take unitary operators $S,T,S',T'\in\bb$ such that $SaT=\sum_{k=1}^ns_kP_k$ and 
$S'bT'=\sum_{l=1}^mt_lQ_k$, where $P_k; k=1,...,n$, as well as $Q_l; l=1,...,m$, is a family of 
pairwise orthogonal rank 1 projections. Then $\|a\|_p=(\sum_{k=1}^ns_k^p)^{\frac{1}{p}}$ and 
$\|b\|_p=(\sum_{l=1}^mt_l^p)^{\frac{1}{p}}$. Further, $(S\di S')(a\di b)(T\di 
T')=(\sum_{k=1}^ns_kP_k)\di(\sum_{l=1}^mt_lQ_k)$. Since $S\di S',T\di T'$ are 
unitary operators, 
this implies that 
\[
\|a\di 
b\|_p=\|[(\sum_{k=1}^ns_kP_k)\di(\sum_{l=1}^mt_lQ_k)]\|_p=\|\sum_{k,l}s_kt_lP_k\di Q_l\|_p.
\]
 But, 
since all $P_k\di Q_l$ are pairwise orthogonal rank 1 projections, the  last number is 
$(\sum_{k,l}(s_kt_l)^p)^{\frac{1}{p}}= 
(\sum_{k=1}^ns_k^p)^{\frac{1}{p}}(\sum_{l=1}^mt_l^p)^{\frac{1}{p}}=\|a\|_p\|b\|_p$. 
\end{proof}
\begin{proposition}
For every $p,q\in[1,\ii]$ and $r:=\max\{p,q\}$ the bioperator
$\rr:{^{(p)}}E\times{^{(q)}}F\to{^{(r)}}(E\ot_{pr}F):(x,y)\mt x\ot y$ is completely contractive. 
\end{proposition} 

\begin{proof}
Take $u\in\cK({^{(p)}}E), v\in\cK({^{(q)}}F)$ and choose $\e>0$. By definition of `$\ot_{pr}$', 
there exist representations of $u$ as $\sum_{k=1}^na_kx_k$ and $v$ as $\sum_{l=1}^mb_ly_l$   such 
that $\sum_{k=1}^n\|a_k\|_p\|x_k\|<\|u\|_{pr}+\e$ and $\sum_{l=1}^m\|b_l\|_q\|y_l\|<\|v\|_{pr}+\e$. 
We have $\rr_\ii(u,v)=\sum_{k,l}a_k\di b_l(x_k\ot y_l)$; therefore, by the previous proposition,  
\[
\|\rr(u,v)\|\le\sum_{k,l}\|a_k\di 
b_l\|_r\|x_k\|\|y_l\|\le\left(\sum_{k=1}^n\|a_k\|_r\|x_k\|\right)
\left(\sum_{l=1}^m\|b_l\|_r\|y_l\|\right)\le
\]
\[ 
\left(\sum_{k=1}^n\|a_k\|_p\|x_k\|\right)
\left(\sum_{l=1}^m\|b_l\|_q\|y_l\|\right)< (\|u\|_{pr}+\e)(\|v\|_{pr}+\e).
\] 
Since $\e$ is arbitrary, this implies  that $\|\rr_\ii(u,v)\|\le\|u\|\|v\|$. 
\end{proof} 

\section{The projective tensor product ` ${\ot_{pop}}$ ', its definition and the existence theorem} 
A widespread point of view, inherited from pure algebra, is that the raison d'etre of a `good' 
tensor product is that it linearizes some respective `good' class of bioperators (cf.\cite[pp. 
3-5]{bou}). As to the theory of $PG$-- ( = matricially normed) spaces, one could show that the 
Haagerup tensor product, introduced in~\cite{er1}, linearizes what was called in~\cite{efr} 
multiplicatively bounded bioperators. But this is outside the scope of the present paper. Here we 
shall introduce another, `projective' tensor product of $PG$--spaces that linearizes what was 
called in the cited textbook, as well as in this paper, completely bounded bioperators. 

Fix, for a time, two normed $PQ$--spaces $E$ and $F$. 

\begin{definition} 
A pair $(\Theta,\theta)$, consisting of a normed $PQ$--space $\Theta$ and a completely contractive 
bioperator $\theta:E\times F\to\Theta$, is called non-completed {\it proto-operator-projective 
tensor product} of $E$ and $F$  or, for brevity, {\it projective tensor product}  of $E$ and $F$
 if, for every completely bounded bioperator
$\rr:E\times F\to G$, where $G$ is a $PQ$--space, there exists a unique completely bounded operator 
$R:\Theta\to G$ such that the diagram 
\[
\xymatrix@R-10pt@C+15pt{
E\times F \ar[d]^{\theta} \ar[dr]^{\rr} & \\
\Theta \ar[r]^R  &  G  } 
\]
\noindent is commutative, and moreover $\|R\|_{cb}=\|\rr\|_{cb}$. 
\end{definition}
Uniqueness, in a proper sense, of such a pair is a particular case of the general-categorical 
observation, concerning the uniqueness of the initial object in a category (cf.\cite{mcl}). We 
shall prove the existence of such a pair, displaying its explicit construction. 

First, we need an additional version of the operation ` $\di$ ', this time connecting 
 elements  of amplifications. Namely, for $u\in \cK E,v\in\cK F$ we set 
 $u\di v:= \vartheta_\ii(u,v)\in\cK(E\ot F)$, where $\vartheta:E\times F\to E\ot 
F:(x,y)\mt x\ot y$ is the canonical bilinear operator. In particular, for elementary tensors we 
have $ax\di by:=(a\di b)(x\ot y).$ 

Note that for all $a,b,c,d\in\bb$ and $u\in \cK E,v\in\cK F$ we have 
\begin{align}
(a\di b)\cd(u\di v)\cd(c\di d)=(a\cd u\cd c)\di(b\cd v\cd d).
\end{align}
One can immediately verify this formula on elementary tensors. 

It is easy to show that every $U\in\cK(E\ot F)$ can be represented as 
\begin{align}
\sum_{k=1}^na_k\cd(u_k\di v_k)\cd b_k
\end{align}
\noindent for some $a_k,b_k\in\bb, u_k\in\cK E,v_k\in\cK F, k=1,...,n$ 
(see details in~\cite[Section 7.2]{heb2}). This implies that the operator $\bb\ot\cK E\ot\cK 
F\ot\bb\to\cK(E\ot F)$, associated with the 
4-linear operator $(a,u,v,b)\mt a\cd(u\di v)\cd b$, is surjective. Thus $\cK(E\ot F)$ can be 
endowed with the semi-norm of the respective quotient space of $\bb\ot_{pr}\cK E\ot_{pr}\cK 
F\ot_{pr}\bb$, denoted by $\|\cd\|_{pop}$. In other words, for $U\in\cK(E\ot F)$ we have 
\begin{align}
\|U\|_{pop}:=\inf\{\sum_{k=1}^n\|a_k\|\|u_k\|\|v_k\|\|b_k\|\},
\end{align}
where the infimum is taken over all possible representations of $U$ in the form given by (5.2). 

Now observe that $\bb\ot_{pr}\cK E\ot_{pr}\cK F\ot_{pr}\bb$ is a contractive $\bb$-bimodule, being 
considered as a tensor product of the left $\bb$-module $\bb$ with the linear space $\cK E\ot\cK F$ 
and the right $\bb$-module $\bb$. Therefore $\cK(E\ot F)$ is the image of a contractive 
$\bb$-bimodule with respect to a quotient map of semi-normed spaces. Since the latter map is 
obviously a bimodule morphism, we easily obtain that the bimodule $(\cK(E\ot F),\|\cd\|_{pop}$ is 
also contractive. 

We see that $\|\cd\|_{pop}$ is a $PQ$--semi-norm on $E\ot F$. Denote the respective semi-normed 
$PQ$--space by $E\ot_{pop}F$. 

Finally, note that if $\rr:E\times F\to G$ is a bioperator, and $R:E\ot F\to G$ is the associated 
linear operator, then we obviously have the formula 
\begin{align}
R_\ii(u\di v)=\rr_\ii(u,v).
\end{align}
\begin{theorem}
 {\rm (Existence theorem)}. The pair $(E\ot_{pop}F,\vartheta)$  is a non-completed projective tensor 
 product of $E$ and $F$.
\end{theorem}
We prefer to give a self-contained proof of the theorem, despite some of its parts (not all) 
resemble with what was said in~\cite{heb2} under the assumption that we deal with quantum spaces. 
\begin{proof}
First, for arbitrary $u\in\cK E, v\in\cK F$ we have, of course, $u\di v=\id\cd(u\di v)\cd\id$. 
Therefore the bioperator $\vartheta$, considered with values in $E\ot_{pop}F$, is completely 
contractive, or, equivalently, we have 

\begin{align}
\|u\di v\|_{pop}\le\|u\|\|v\|.
\end{align}
%

Now let $G$ be a $PQ$--space, $\rr:E\times F\to G$  a completely bounded bioperator, and 
$R:E\ot_{pop}F\to G$ the associated linear operator. We want to show that $R$ is completely bounded 
and that $\|\rr\|_{cb}=\|R\|_{cb}$. 

Take $U\in\cK(E\ot_{pop}F)$ and represent it as in (5.2). Since $R_\ii$ is a $\bb$-bimodule 
morphism, we have by (5.4) that $R_\ii(U)=\sum_{k=1}^na_k\cd\rr_\ii(u_k, v_k)\cd b_k$, hence 
$\|R_\ii(U)\|\le\|\rr\|_{cb}\sum_{k=1}^n\|a_k\|\|u_k\|\|v_k\|\|b_k\|$. Therefore the definition of 
$\|\cd\|_{pop}$ implies that $R$ is completely bounded, together with $\rr$, and 
$\|R\|_{cb}\le\|\rr\|_{cb}$. The inverse estimate follows from the inequality 
$\|\rr_\ii(u,v)\|\le\|R_\ii\|\|u\|\|v\|$, which, in its turn, immediately follows from  (5.4) and 
(5.5). 

Now consider the diagram from Definition 5.1 with  $E\ot_{pop}F$ and $\vartheta$ in the capacity of 
$\Theta$ and $\theta$, respectively. It is known from linear algebra, that $R$ is the only  linear 
operator, making the diagram commutative. Thus we see that the pair $(E\ot_{pop}F,\vartheta)$ 
satisfies almost all requirements given in Definition 5.1. The only remaining thing is to show that 
the semi-norm $\|\cd\|_{pop}$ is actually a norm. 

By Proposition 2.4, for this aim it is sufficient to show that, for every non-zero elementary 
tensor $Qw$, where $Q$ is a rank 1 operator of norm 1 and $w\in E\ot_{pop}F, w\ne0$, we have 
$\|Qw\|_{pop}\ne0$. Since $E$ and $F$ are {\it normed} spaces, there exist bounded functionals 
$f:E\to\co, g:F\to\co$ such that for $f\ot g:E\ot F\to\co$ we have $(f\ot g)w\ne0$. As we know from 
the previous section, the bilinear functional $\rr:=f\times g:E\times F\to{^{(\ii)}}\co$  is 
completely bounded. Therefore, choosing $G:={^{(\ii)}}\co$, we see that the associated linear 
functional, that is $f\ot g:E\ot_{pop}F\to{^{(\ii)}}\co$, is also completely bounded. Hence, we 
have 
$$
|(f\ot g)w|=\|Q[(f\ot g)(w)]\|=\|[(f\ot g)_\ii(Qw)\|\le\|f\ot g\|_{cb}\|Qw\|_{pop}.
$$
 Therefore $\|Qw\|_{pop}\ne0$ since $(f\ot g)w\ne0$.
\end{proof}
\bigskip
Note that in the underlying space of $E\ot_{pop} F$ we have 
\begin{align}
\|x\ot y\|\le\|x\|\|y\| \q {\rm for \q all}\q x\in E,y\in F.
\end{align}
Indeed, take two operators $P,Q\in\cK$ of rank 1 and of norm 1. We see that $\|P\di Q\|=1$ and that 
$P\di Q$ has also rank 1. Therefore, $\|x\ot y\|=\|(P\di Q)x\ot y\|_{pop}=\|Px\di Qy\|$. It remains 
to use (5.5). 

 (In fact, in (5.6), as well as in (5.5), we have the exact equality, but we shall not discuss it now).

\bigskip
So far, we spoke about general (normed) $PQ$--spaces. But their tensor product has a natural 
analogue in the context of {\it complete} or {\it Banach $PQ$--spaces}. The latter are, by 
definition, $PQ$--spaces with complete underlying normed spaces. As in the `classical' context, for 
every $PQ$--space $E$ there exists its {\it completion}, which is defined as a 
 pair $(\overline{E},i:E\to\overline{E})$, consisting of a complete $PQ$--space and a
completely isometric operator, such that the same pair, considered for respective underlying spaces 
and operators, is the `classical' completion of $E$ as of a normed space. The proof of the 
respective existence theorem repeats word by word the simple argument given in~\cite[Chapter 
4]{heb2} for $Q$--spaces. We only recall that the norm on $\cK\overline{E}$ is introduced with the 
help of the natural embedding of $\cK\overline{E}$ into $\overline{\cK E}$, the `classical' 
completion of $\cK E$. 

 It is easy to observe that the characteristic universal property of the `classical' completion has its 
 proto-quantum version ({\it ibidem}). Namely, if $(\overline{E}, i)$ is the
  completion of a $PQ$--space $E, F$ a $PQ$--space and $\va: E\to F$ is a completely bounded operator, 
  then there exists a unique completely bounded operator
  $\overline{\va}:\overline{E}\to\overline{F}$ that extends, in the obvious sense, $\va$. Moreover, we 
  have $\|\overline{\va}\|_{cb}=\|{\va}\|_{cb}$.

Let us distinguish the following fact that will be useful. Its proof repeats word by word the 
argument in Proposition 4.8 in~\cite{heb2}. 

\begin{proposition} \label{context}
Let $\va:E\to F$ be a completely isometric isomorhism between $PQ$--spaces. Then its continuous 
extension $\overline{\va}:\overline{E}\to\overline{F}$ is also a completely isometric isomorhism. 
\end{proposition}

Now we can speak of the {\it completed} projective tensor product of two 
$PQ$--spaces. Its definition repeats Definition 5.1, but, what is essential, with the following 
difference: 
{\it $\Theta$ and $G$ are supposed to be complete}.  Using the universal property of the 
completion, we immediately see that {\it the completed projective tensor product of 
$PQ$--spaces $E$ and $F$ exists: it is the pair $(E\widehat\ot_{pop} F,\widehat\vartheta)$, where 
$E\widehat\ot_{pop}F $ is the completion of the $PQ$--space $E\ot_{pop} F$, and $\widehat\vartheta$ 
acts as $\vartheta$, but with range $E\widehat\ot_{pop} F $}.

\section{Tensoring by $L_1(\cd)$, and some other computation}
In this section we show that for certain concrete tensor factors their projective tensor product 
also becomes something concrete and transparent. We shall see that the behavior of this tensor 
product resembles the behavior of the projective tensor product in the classical context. 

Denote the completion of the $PQ$-space $E\ot_{pr} F$ from Example \ref{ex8} by $E\widehat\ot_{pr} 
F$. Clearly, it is a $P$--quantization of the `classical' completed projective tensor product, 
denoted also by $E\widehat\ot_p F$; it will not create a confusion.

\begin{theorem} 
Let  $E$ be a normed space, $F$ a $PQ$--space, $p\in[1,\ii]$, ${^{(p)}}E$ the $PQ$--space from 
Example 2.5, and $E\ot_{pr}F$ the $PQ$--space from Example 4.2. Suppose that $F$ is $p$--convex. 
Then there exists a completely isometric isomorphism $I: {^{(p)}}E\ot_{pop}F\to E\ot_{pr}F$, acting 
as the identity operator on the common underlying linear space of our $PQ$--spaces. As a corollary 
{\rm (see Proposition 5.3),} there exists a  completely isometric isomorphism  $\widehat I : 
{^{(p)}}E\widehat\ot_{pop}F\to E\widehat\ot_{pr}F$, which is the extension by continuity of $I$. 
\end{theorem}
\begin{proof}
Consider the canonical bioperator $\vartheta:{^{(p)}}E\times F\to E\ot_{pr}F$. Since the 
$PQ$--space is $p$--convex, it gives rise, by virtue of Proposition \ref{pr}, to the completely 
contractive operator $I$, acting as in the formulation. 
 Therefore it is sufficient to show that 
 for every $U\in\cK(E\ot F)$ its norm in $\cK[{^{(p)}}E\ot_{pop}F]$ is not greater than 
 $\|I_\ii(U)\|$ or, equivalently, than the norm of $\beta(U)$ in $E\ot_{pr}\cK F$. 

Fix $U$ and choose $\varepsilon>0$; then there exists a representation $\beta(U)=\sum_{k=1}^n 
x_k\ot v_k; x_k\in E, v_k\in\cK F$ such that $\sum_{k=1}^n\|x_k\|\|v_k 
\|<\|\beta(U)\|+\varepsilon=\|I_\ii(U)\|+\varepsilon$. 

Now choose an arbitrary rank one projection $P\in\cK$ and set $V:=\sum_{k=1}^nPx_k\di 
v_k\in\cK(E\ot F)$. By (4.1), there exists an isometry $S\in\bb$ such that $S^*(P\di a)S=a$ for 
every $a\in\cK$. Representing every $v_k$ as a sum of elementary tensors, we easily see that 
$\beta(S^*\cd V\cd S)=\sum_{k=1}^n x_k\ot v_k=\beta(U)$. Therefore $U=S^*\cd V\cd S$, hence 
$\|U\|_{pop}\le\|V\|_{pop}$. But by (5.5) we have 
$\|V\|_{pop}\le\sum_{k=1}^n\|Px_k\|\|v_k\|=\sum_{k=1}^n\|x_k\|\|v_k\|$, and consequently 
$\|U\|_{pop}\le\|I_\ii(U)\|+\varepsilon$. Since  such an estimate holds for every $\varepsilon>0$, 
we are done. 
\end{proof}

\begin{corollary}
 With $p$ and $F$ as above, there exists a completely isometric isomorphism 
$I: {^{(p)}}\co\ot_{pop}F\to F$, acting as $\lm\ot x\mt\lm x$. 
\end{corollary} 

In its turn, this assertion, since ${^{(q)}}\co$ is $p$--convex provided $q<p$, implies
\begin{corollary}
${^{(p)}}\co\ot_{pop}{^{(q)}}\co={^{(r)}}\co$, where $r=\max\{p,q\}$. 
\end{corollary} 

\begin{remark}
The projective tensor product of two ${\cal L}^p$--spaces is not bound to be again an ${\cal 
L}^p$--space. Indeed, consider the projective tensor square of the ${\cal L}^2$--space 
$\ell_2({^{(2)}}\co)$ and the elements $Pe_1,Qe_2\in\cK(\ell_2({^{(2)}}\co))$, where $P,Q$ are 
orthogonal projections in $\cK$ (i.e. $PQ=0$), and $e_1,e_2$ are orts in $\ell_2$. Then it is not 
difficult to show that, despite our elements have orthogonal supports, we have that 
$\|Pe_1+Qe_2\|_{pop}=2$, whereas $(\|Pe_1\|^2+\|Qe_2\|^2)^{\frac{1}{2}}=\sqrt{2}$. Incidentally, it 
is shown in~\cite{hend} that there exists a kind of projective tensor product in the class of 
 the so-called $p$--convex $p$--multi-spaces (see~\cite{dlot} and 
also~\cite{lam}), reflecting their special properties. Therefore
one may suggest that $p$--convex $PQ$--spaces have their own projective tensor product, defined 
only within that class and not leading out of this class. 
\end{remark}
In the following proposition we deal, generally speaking, with $PQ$--spaces that are not 
$p$--convex. 

\begin{theorem} 
Let  $E$ and $F$ be normed spaces, $p\in[1,\ii]$ 
Then there exists a completely isometric isomorphism $I: 
{^{(p)}}E\ot_{pop}{{^{(p)}}F}\to{^{(p)}}(E\ot_{pr}F)$, acting as the identity operator on the 
common underlying linear space of our $PQ$--spaces. As a corollary {\rm (see Proposition 5.3) } 
there exists a  completely isometric isomorphism $\widehat I: 
{^{(p)}}E\widehat\ot_{pop}{{^{(p)}}F}\to {^{(p)}}(E\widehat\ot_{pr}F)$, which is the extension by 
continuity of $I$. 
\end{theorem}
\begin{proof}
Consider the canonical bioperator 
$\vartheta:{^{(p)}}E\times{^{(p)}}F\to{^{(p)}}(E\widehat\ot_{pr}F)$ . It gives rise, by virtue of 
Proposition 4.7, to the completely contractive operator $I$, acting as in the formulation.  
Therefore it is sufficient to show that 
 for every $U\in\cK({^{(p)}}E\ot_{pop}{^{(p)}}F)$ we have $\|U\|_{pop}\le\|I_\ii(U)\|$.

Fix $U$ and choose $\varepsilon>0$. As a normed space, $\cK[{^{(p)}}(E\ot_{pr}F)]$ is 
$\cK_p\ot_{pr}E\ot_{pr}F$. Consequently, in the linear space $\cK(E\ot F)=\cK\ot E\ot F$ there 
exists a representation of $U$ or, which is the same, of $I_\ii(U)$ as $\sum_{k=1}^na_k\ot x_k \ot 
y_k; a_k\in\cK, x_k\in E, y_k\in F$ such that 
$\sum_{k=1}^n\|a_k\|_p\|x_k\|\|y_k\|<\|I_\ii(U)\|+\varepsilon$. 

Take an arbitrary rank 1 projection $P\in\cK$ and introduce an element $V:=\sum_{k=1}^nPx_k\di(a_k 
y_k)=\sum_{k=1}^n(P\di a_k)x_k\ot y_k\in\cK(E\ot F)$. As we know by (4.1), there exists an isometry 
$S\in\bb$ such that $a_k=S^*(P\di a_k)S$ for all $k$. It follows that $S^*\cd V\cd 
S=\sum_{k=1}^na_k(x_k \ot y_k)=U$. Therefore, for $U$ and $V$ as of elements of 
$\cK[{^{(p)}}E\widehat\ot_{pop}{{^{(p)}}F}]$, we have $\|U\|_{pop}\le\|V\|_{pop}$. But, by (5.5), 
we have $\|V\|_{pop}\le \sum_{k=1}^n\|Px_k\|\|a_k y_k\|$, where norms are taken in $\cK{^{(p)}}E$ 
and $\cK{^{(p)}}F$, respectively. Further, $\|Px_k\|=\|x_k\|$ and, since 
$\cK{^{(p)}}F=\cK_p\ot_{pr}F$, we have $\|a_k y_k\|=\|a_k\|_p\|y_k\|$. Consequently, we have 
$\|U\|_{pop}\le\sum_{k=1}^n\|a_k\|_p\|x_k\|\|y_k\| <\|I_\ii(U)\|+\varepsilon$. 
 Since such an estimate holds for every $\varepsilon>0$, we are done. 
\end{proof} 
Setting in this theorem $p:=1$, we obtain 
\begin{corollary} 
For all normed spaces $E$ and $F$ we have, up to a completely isometric isomorphism, 
$E_{\max}\ot_{pop}F_{\max}=(E\ot_{pr}F)_{\max}$ and \\
$E_{\max}\widehat\ot_{pop}F_{\max}=(E\widehat\ot_{pr}F)_{\max}$. 
\end{corollary}

As another particular case, we see that for for a Hilbert space $H$ we have \\ ${^{(p)}}H 
\widehat\ot_{pop}{^{(p)}}H={^{(p)}}{\cal N}(H)$, where ${\cal N}(H)$ is the Banach space of trace 
class operators on $H$. 

By virtue of Grothendieck Theorem, mentioned in Introduction, we may say that in the case of the 
classical projective tensor product of normed spaces, the especially nice tensor factors are 
$L_1$--spaces. Now we would like to show that the same is true for the projective tensor product of 
$PG$--spaces. At first we need some preparation. 

\begin{proposition}
 Let $(X,\mu),(Y,\nu)$ be measure spaces, and $E,F$ be $PQ$--spaces, $p\in[1,\ii]$. Then the bioperator 
 $\rr:L_p(X,E)\times L_p(Y,F):L_p(X\times Y,E_{pop}F):(\bar x,\bar 
y)\mt\bar z;\bar z(s,t):=x(s)\ot y(t) $ is completely contractive. 
\end{proposition} 

(Here and thereafter, speaking about $L_p(X,\cd)$ and $L_p(Y,\cd)$, we mean $X$ and $Y$ with the 
given measures, and speaking about $L_p(X\times Y,\cd)$, we consider $X\times Y$ with the cartesian 
product of these measures). 

\begin{proof}
Consider the diagram 
\[
 \xymatrix @C+20pt{\cK(L_p(X,E))\times\cK(L_p(Y,F)) \ar[r]^{\rr_\ii}\ar[d]_{\al_X\times\al_Y} 
& \cK(L_p(X\times Y,E\ot_{pop}F)) \ar[d]^{\al_{X\times Y}} \\
L_p(X,\cK E)\times L_p(Y,\cK F) \ar[r]^{{\cal S}} & L_p(X\times Y,\cK E_{pop}F) }, 
\]
where ${\cal S}$ takes a pair of vector--functions $(\bar u(s),\bar v(t))$ to the vector--function 
$\bar w(s,t):=\bar u(s)\di\bar v(t); s\in X, t\in Y$, and $\al_X$ etc. are the respective 
specializations of $\al$ from Example 4.1. The diagram is evidently commutative, and $\al$ is an 
isometry. Therefore it suffices to show that ${\cal S}$ is contractive. Indeed, with the help of 
the estimate (5.5), we have 
$$
\|\bar w\|=\left(\int_{X\times 
Y}\|\bar u(s)\di\bar v(t\|^pd(s,t)\right)^{\frac{1}{p}}\le
\left(\int_{X\times Y}\|\bar u(s)\|^p\|\bar 
v(t)\|^p)d(s,t)\right)^{\frac{1}{p}}=\|\bar u\|\|\bar v\|.
$$ 
 \end{proof}
%
%

\begin{theorem}
Let $(X,\mu),(Y,\nu)$ be measure spaces, $E,F$ $PQ$--spaces. Then there exists the complete 
isometry $I: L_1(X,E)\ot_{pop}L_1(Y,F)\to L_1(X\times Y,E\ot_{pop}F)$, well defined by $\bar 
x\ot\bar y\mt \bar z$, where $\bar z(s,t):=\bar x(s)\ot\bar y(t)$. 
\end{theorem}

\begin{proof}
 The bioperator $\rr$ from Proposition 6.7, being considered for $p=1$, gives rise to 
the completely contractive operator $R$, acting exactly as $I$ in the formulation. Therefore our 
task is to show that for every $U\in\cK[L_1(X,E)\ot_{pop} L_1(Y,F)]$ we have 
$\|U\|_{pop}\le\|R_\ii(U)\|$. 

Obviously, $L_1(X,E)$ contains the dense subspace $L_1^0(X,E)$ consisting of 
 vector-functions of the form $\sum_k\chi_k x_k$, where $\chi_k$ are characteristic functions of 
 pairwise disjoint
subsets of $X$ with finite measure, and $x_k\in E$. Similarly, $L_1(Y,F)$ contains the subspace 
$L_1^0(Y,F)$ with analogues properties. Therefore, thanks to the estimate (5.6) and the last 
estimate in Proposition 2.6, it is sufficient to prove that $R_\ii$ does not decrease norms of sums 
of elementary tensors of the form $a({\bf x}\ot{\bf y})$, where $a\in\cK,{\bf x}\in L_1^0(X,E),{\bf 
y}\in L_1^0(Y,F)$. 

Let $U$ be such a sum. It is not difficult to show that it can be represented as 
\[
U=\sum_{k=1}^n\sum_{l=1}^ma_{kl}(\bar x_k\ot\bar y_l),
\] 
with $\bar x_k(s):=\chi_k(s)x_k; s\in X$, where 
$\chi_k(s)$ are characteristic 
functions of pairwise disjoint subsets $X_k;\mu(X_k)<\ii, x_k\in E$, and $\bar 
y_l(t):=\chi'_l(t)y_l; t\in Y$, where 
$\chi'_l(t)$ are characteristic functions of 
pairwise disjoint subsets $Y_l;\nu(Y_l)<\ii, y_l\in F$. 

We obviously have $R_\ii(U)=\sum_{k,l}a_{kl}[\chi_k(s)\chi'_l(t)x_k\ot y_l]$. Therefore, by the 
recipe of Example 4.1, $\|R_\ii(U)\|$ is the norm of the function \\ 
$\sum_{k,l}\chi_k(s)\chi'_l(t)[a_{kl}(x_k\ot y_l)]\in L_1(X\times Y),\cK(E\ot_{pop}F)$. Since we 
are in a space $L_1(\cd)$, this implies that 
\begin{align} 
\|R_\ii(U)\|=\sum_{k=1}^n\sum_{l=1}^m\mu(X_k)\nu(Y_l)\|a_{kl}(x_k\ot y_l)\|_{pop}.
\end{align}
Fix, for a time, a pair $k,l$, and also $\e>0$. By (5.3), we can represent $a_{kl}(x_k\ot 
y_l)\in\cK(E\ot_{pop}F)$ in the form $\sum_ib_{kl}^i\cd(u_{kl}^i\di v_{kl}^i)\cd c_{kl}^i$, where 
$b_{kl}^i,c_{kl}^i\in\bb,u_{kl}^i\in\cK E,v_{kl}^i\in\cK F$, such that 
\begin{align} 
\sum_i\|b_{kl}^i\|\|u_{kl}^i\|\|v_{kl}^i\|\|c_{kl}^i\|<\|a_{kl}(x_k\ot y_l)\|_{pop}+\e. 
\end{align}

(Here the number of summands, indexed by $i$, of course, depends on the pair $k,l$). 

Now for every $u\in\cK E$ we denote, for brevity, by $\bar u\in\cK L_1(X,E)$ the element 
$B_\ii(u)\in\cK(L_1(X,E))$, where $B:E\to L_1(X,E)$ takes $x$ to the vector-function $\bar 
x(s):=\chi_k(s)x$. Similarly, for $v\in L_1(Y,F)$ we set $\bar v:=B'_\ii(v)\in\cK(L_1(Y,F))$, where 
$B':F\to L_1(Y,F):y\mt\bar y:=\chi'(t)y$. Evidently we have 

\begin{align}
\|\bar u\|=\mu(X_k)\|u\| \q {\rm and} \q \|\bar v\|=\nu(Y_k)\|v\|. 
\end{align}


Further, since for $x\in E,y\in F$ we have $R(\bar x\ot\bar y)=\chi_k(s)\chi'(t)x\ot y$, it easily 
follows that for $u\in\cK E,v\in\cK F$ we have $R_\ii(\bar u\di\bar v)=\chi_k(s)\chi'(t)u\di v$. 
Consequently, taking in account (6.3) and (6.2), we have 
$$
R_\ii\left(\sum_ib_{kl}^i\cd(\bar u_{kl}^i\di\bar v_{kl}^i)\cd c_{kl}^i\right)=\chi_k(s)\chi'(t) 
\sum_ib_{kl}^i\cd(u_{kl}^i\di v_{kl}^i)\cd c_{kl}^i=
$$
$$
\chi_k(s)\chi'_l(t)a_{kl}(x_k\ot y_l)=R_\ii[a_{kl}(\bar x_k\ot\bar y_l)]. 
$$
But $R$ is obviously injective and, of course, the same is true for $R_\ii$. It follows that 
$\sum_ib_{kl}^i\cd(\bar u_{kl}^i\di\bar v_{kl}^i)\cd c_{kl}^i=a_{kl}(\bar x_k\ot\bar y_l)$, and 
consequently $U=\sum_{k,l}[\sum_ib_{kl}^i\cd(\bar u_{kl}^i\di\bar v_{kl}^i)\cd c_{kl}^i]$. This, 
with the help of (6.3), implies that 
$$
\|U\|_{pop}\le\sum_{k,l}\sum_i\|b_{kl}^i\|\|\bar u_{kl}^i\|\|\bar 
v_{kl}^i\|\|c_{kl}^i\|=\sum_{k,l}\mu(X_k)\nu(Y_l)\sum_i\|b_{kl}^i\|\|u_{kl}^i\|\|v_{kl}^i\|\|c_{kl}^i\|.
$$ 
From this, by virtue of (6.2), we obtain that 
$$
\|U\|_{pop}\le\sum_{k,l}\mu(X_k)\nu(Y_l)(\|a_{kl}(x_k\ot y_l)\|_{pop}+\e).
$$
 Since $\e$ is arbitrary, it follows that 
$\|U\|_{pop}\le\sum_{k,l}\mu(X_k)\nu(Y_l)\|a_{kl}(x_k\ot y_l)\|_{pop}$, that is, by (6.1), 
$\|U\|_{pop}\le\|R_\ii(U)\|$. 
\end{proof}

Since $L_1(X,E)\ot_{pop}L_1(Y,F)$ is dense in $L_1(X,E)\widehat\ot_{pop}L_1(Y,F)$, and the image of 
$I$ is obviously dense in $L_1(X\times Y,E\widehat\ot_{pop}F)$, we have, as an immediate corollary, 

\begin{theorem}
 Let $(X,\mu),(Y,\nu),E,F$ be as before. Then there exists a complete 
isometric isomorphism $I: L_1(X,E)\widehat\ot_{pop}L_1(Y,F)\to L_1(X\times Y,E\widehat\ot_{pop}F)$, 
well defined by $\bar x\ot\bar y\mt \bar z$, where $\bar z(s,t):=\bar x(s)\ot\bar y(t)$ 
\end{theorem}

Combining Theorem 6.8 or 6.9 with the previous results in this section,
one can obtain various corollaries. For example, taking the one-point $Y$ and using Corollary 6.2, 
we get the assertion that can be considered as a `$PQ$--version' of the Grothendieck Theorem in its 
usual formulation: 

\begin{corollary} 
Let $p\in[1,\ii], X$ be a measure space, and $F$ be a complete $p$--convex $PQ$--space. Then we 
have, up to a completely isometric isomorphism, $L_1(X,{^{(p)}}\co)\widehat\ot_{pop}F=L_1(X,F)$. 
\end{corollary}

Note that the same assertion could be obtained without using Theorem 6.8, by combining easier 
Proposition 4.4 with Corollary 6.2. 

Also, combining Theorem 6.9 (or Proposition 4.4) with Corollary 6.3, one can get the completely 
isometric isomorphism $L_1(X,{^{(p)}}\co)\widehat\ot_{pop}L_1(Y,{{^{(q)}}}\co)\simeq L_1(X\times 
Y,{^{(r)}}\co)$ with $r:=\max\{p,q\}$, and so on. 


\section{Quantum duality and adjoint associativity}
  
We proceed to show that the projective tensor product of $PQ$--spaces satisfies the law of adjoint 
associativity ( = exponential law), connecting it with the proper $P$--quantization of the space of 
completely bounded operators. Such a $P$--quantization extends what was well known in the context 
of quantum (operator) spaces. In that context the relevant construction cropped up in 
~\cite[p.140]{ro}, but was fully realized and put in proper place independently and simultaneously 
in~\cite{blp} and~\cite{er2}. In the matrix-free language, again only for $Q$--spaces, it was 
presented in~\cite[8.1.8.]{heb2}  Here we give all needed details for general $PQ$--spaces. 

Let $E,G$ be two $PQ$--spaces. Our task is to endow the normed space $\CB(E,G)$ with a 
$P$--quantization. For this aim we consider the {\it evaluation bioperator} $\mathcal{E}\colon 
E\times\CB(E,G)\to G\colon (x,\va)\mapsto\va(x)$ and its  amplification $\mathcal{E}_\ii\colon \cK 
E\times\cK[\CB(E,G)]\to\cK G$; the latter, as we remember, is well defined by 
$(ax,b\va)\mapsto(a\di b)\va(x)$. Set, for $\Phi\in\cK[\CB(E,G)]$, 

\begin{align}
\|\Phi\|:=\sup\{\|\mathcal{E}_\ii(u,\Phi)\|; u\in\cK E;\; \|u\|\le1\}.
\end{align}

(We see that such a definition closely imitates the definition of the `classical' operator norm: 
indeed, $\|\va\|$ is $\sup\{\|\mathcal{E}(x,\va)\|; x\in E;\; \|x\|\le1\}$, where $\mathcal{E}:
E\times\bb(E,G)\to G$ is the obvious `classical' evaluation operator).

\begin{proposition}
 The function $\Phi\mt\|\Phi\|$ is a PQ--norm on $\CB(E,G)$, and the resulting 
$PQ$--space is a $P$--quantization of $\CB(E,G)$ as of a normed space. 
\end{proposition}

\begin{proof}
For every $b\in\cK$, $u\in\cK E$ and $\va\in\CB(E,G)$ we obviously have the equality 
\begin{align}
 {\cal E}_\ii(u,{b\va})=\va_\ii(u)\di b.
 \end{align} 

Therefore, by the Proposition 4.3(iii), we have $\|{\cal 
E}_\ii(u,{b\va})\|\le\|b\|_1\|\|\va_\ii(u)\|$. 
 It follows that the number $\|b\va\|$ is well-defined. Consequently, proceeding from elementary 
 tensors to 
 their (finite) sums, that the number $\|\Phi\|$ is well-defined for all $\Phi\in\cK\CB(E,G)$ . 

 Further, for all $a\in\cK, u\in\cK E,\Phi\in\cK[\CB(E,G)]$ we have the equalities 
$$
{\cal E}_\ii(u,a\cd\Phi)=(\id\di a)\cd{\cal E}_\ii(u,{\Phi}) \q {\rm and} \q
{\cal E}_\ii(u,\Phi\cd a)={\cal E}_\ii(u,{\Phi})\cd(\id\di a)
$$
that can be immediately checked on elementary tensors. Consequently, \\ $\|{\cal 
E}_\ii(u,a\cd\Phi)\|,\|{\cal E}_\ii(u,\Phi\cd a)\|\le\|\id\di a\|\|{\cal 
E}_\ii(u,{\Phi})\|=\|a\|\|{\cal E}_\ii(u,{\Phi})\|$. It follows that $\|a\cd 
\Phi\|\le\|a\|\|\Phi\|$, and similarly $\|\Phi\cd a\|\le\|a\|\|\Phi\|$. Therefore the introduced 
seminorm on $\cK[\CB(E,G)]$ is a $PQ$--seminorm on $\cK[\CB(E,G)]$. 

Finally, take a rank one projection $P\in\cK$. Then, considering $\CB(E,G)$ as the underlying 
semi-normed space of the introduced $PQ$--space, we have, by virtue of (7.2) and Proposition 
4.3(ii), that 
$$
\|\va\|=\sup\{\|\mathcal{E}_\ii(u,P\va)\|; \|u\|\le1\}=
\sup\{\|\va_\ii(u)\di P\|; \|u\|\le1\}=
$$
$$
\sup\{\|\va_\ii(u)\|; \|u\|\le1\}=\|\va\|_{cb}.
$$
Consequently, our underlying space is just $\CB(E,G)$ with its $cb$--norm. Therefore, by 
Proposition 2.4, the seminorm on $\cK[\CB(E,G)]$ is actually a norm, and the respected $PQ$--space 
is a $P$--quantization of the given space of completely bounded operators. 
\end{proof} 


If a a given normed space $E$ is endowed with a quantization (not just a $P$--quantization), then 
there is a well known standard way to make its dual space $E^*$ again a $Q$--space.
Namely, if we identify the normed spaces $E^*$ and $\CB(E,{^{(\ii)}}\co)$ (that is, if we consider 
$\co$ with its unique quantization), then the recipe above provides the $PQ$--norm on $E^*$ which 
is in fact a $Q$--norm (see, e.g.,~\cite[Ch.8.2]{heb2}). However, if we shall consider other 
$P$--quantizations of $\co$, the normed space $\CB(E,\co)$ is not bound to be the dual of $E$. 
Actually, we already know this: by Proposition 3.2 and 3.3, for $E:={^{(p)}}\co$ the space 
$\CB(E,{^{(q)}}\co)$ is $E^*$, that is just $\co$, if, and only if $p\le q$; otherwise, it is 0. In 
the first case the respected $P$--quantization is as follows. 

\begin{proposition}
If $p\le q$, then the $PQ$--space $\CB({^{(p)}}\co,{^{(q)}}\co)$, after the identification of its 
underlying space with $\co$, is ${^{(q)}}\co$. 
\end{proposition} 

\begin{proof}
In our case every $u\in\cK\co$ has a inique presentation as $a1; a\in\cK, 1\in\co$ whereas every 
$\Phi\in\cK[\CB({^{(p)}}\co,{^{(q)}}\co)]$, after the mentioned identification, 
has a inique presentation as $b1; b\in\cK, 1\in\co$. Consequently, the bioperator $\mathcal{E}_\ii$ 
can be considered as taking $(a1,b1)$ to $(a\di b)1$.
Therefore, by virtue of Proposition 4.6, the $PQ$--norm of a given $\Phi$, presented as $b1$, is 
\[
\|\Phi\|:=\sup\{\|a\di b\|_q:a\in\cK_p\; \|a\|_p\le1\}=
\]
\[
\sup\{\|a\|_q\|b\|_q:a\in\cK_p\; \|a\|_p\le1\}=\|b\|_q\sup\{\|a\|_q:a\in\cK_p\; \|a\|_p\le1\}
\]
But, since $p\le q$, the last supremum is obviously 1.
\end{proof}

Denote the space of completely bounded bioperators from $E\times F$ into $G$ by
$\mathcal{CB}(E\times F,G)$. Obviously, it is the normed space with respect to $\|\cd\|_{cb}$.  

\begin{theorem}
 There exists the isometric isomorphism (of normed spaces) 
$I_F:\mathcal{CB}(E\times F,G)\to\CB(F,\CB(E,G))$, well defined by taking {\rm (exactly as in the 
`classical' context)} the bioperator $\rr$ to the operator $\rr^F\colon F\to\CB(E,G)\colon 
y\mapsto\rr^y$, where $\rr^y:E\to G$ takes $x$ to $\rr(x,y)$. 

To put it in more detailed form, 

(i) for every $\rr\in\mathcal{CB}(E\times F,G)$ and $y\in F$ the operator $\rr^y\colon E\to G$ is 
completely bounded. 

(ii) The operator $\rr^F\colon F\to\CB(E,G)\colon y\mapsto\rr^y$, which is well defined because of 
(i), is completely bounded with respect to the $PQ$--norm on $\CB(E,G)$, defined above. 

(iii) The operator $I_F\colon \mathcal{CB}(E\times F,G)\to\CB(F,\CB(E,G))$, which is well defined 
because of (ii), is an  isometric isomorphism. 
\end{theorem}

\begin{proof}
 First, distinguish the formula $\rr^y_\ii(u)\di b=\rr_\ii(u,by); u\in\cK E, b\in\cK$, easily 
verified on elementary tensors. If our $b$ is a rank 1 projection, it implies, by virtue of 
Proposition 4.3(ii), that $\|\rr^y_\ii(u)\|=\|\rr_\ii(u,by)\|\le\|\rr\|_{\ii}\|u\|\|by\|= 
\|\rr\|_{\ii}\|u\|\|y\|.$ This gives (i). 

Now we may speak about the operator $\rr^F_\ii\colon\cK F\to\cK[\CB(E,G)]$. This time we shall use 
the the formula 
\begin{align} 
 \mathcal{E}_\ii(u,\rr^F_\ii(v))=\rr_\ii(u,v); u\in\cK E, v\in\cK F,
\end{align}
also easily verified on elementary tensors in respective amplifications. Together with (7.1), it 
implies that for $v\in\cK F$ we have 
\[
\|\rr^F_\ii(v)\|=\sup\{\|\mathcal{E}_\ii(u,\rr^F_\ii(v))\|\colon u\in\cK E;\; \|u\|\le1\}=
\]
\[
\sup\{\|\rr_\ii(u,v)\|\colon u\in\cK E;\; \|u\|\le1\}. 
\]

Consequently, $\rr^F_\ii$ is a bounded operator, and we obviously have $\|\rr^F_\ii\|=\|\rr_\ii\|$. 
This gives (ii), and also the equality $\|\rr^F\|_{cb}=\|\rr\|_{cb}$. 

Thus, the operator $I_F\colon \mathcal{CB}(E\times F,G)\to\CB(F,\CB(E,G))$ is well defined and 
isometric. To conclude the proof of the assertion (iii) we shall show that it is surjective. 

\smallskip
Take $\mathcal{S}\in\CB(F,\CB(E,G))$ and set $\rr\colon E\times F\to G\colon 
(x,y)\mapsto[\mathcal{S}(y)](x)$. Clearly, $\rr$ is bounded, and $\rr^F=\mathcal{S}$. Therefore our 
task is to verify that $\rr$ is completely bounded. But the formula (7.3) is obviously valid, if we 
replace $\rr^F$ by ${\cal S}$, hence $\|\rr_\ii(u,v)\|=\|\mathcal{E}_\ii(u,\mathcal{S}_\ii(v))\|$. 
Finaly, it follows from (7.1) that 
$\|\mathcal{E}_\ii(u,\mathcal{S}_\ii(v))\|\le\|u\|\|\mathcal{S}_\ii(v)\|\le\|\mathcal{S}\|_\ii\|u\|\|v\|$, 
and we are done. 
\end{proof}

A similar, up to obvious modifications, argument provides the `twin' isometric isomorphism 
$I_E:\mathcal{CB}(E\times F,G)\to\CB(E,\CB(F,G))$, well defined by taking ( again exactly as in the 
``classical'' context) the bioperator $\rr$ to the operator $\rr^E:x\mt\rr^x$, where $\rr^x:F\to G$ 
acts as $y\mt\rr(x,y)$. 

Now recall that, by virtue of the universal property of the projective tensor product of 
$PQ$--spaces, we can identify the spaces $\CB(E\times F,G)$ and $\CB(E\ot_{pop}F,G)$ by means of 
the isometric isomorphism,
taking a bioperator to its linearization. Therefore, as an immediate corollary of the previous 
proposition, we  obtain the following `proto-quantum' version of the so-called law of adjoint 
associativity in classical functional analysis. (As to the `classical' formulation, see,
e.g.,~\cite[Ch.6.1]{heb2}). 

\begin{theorem}
 There exists an isometric isomorphism \\ 
 $\mathcal{I}_F:\CB(E\ot_{pop}F,G)\to\CB(F,\CB(E,G))$, uniquely 
 determined by the equality 
\[
 ([\mathcal{I}_F(\va)]y)(x)=\va(x\otimes y). 
\]
\end{theorem}

\bigskip
A similar argument provides an isometric isomorphism 
$\mathcal{I}_E:\CB(E\ot_{pop}F,G)\to\CB(E,\CB(F,G))$ by means of the equality 
$([\mathcal{I}_F(\va)]x)(y)=\va(x\otimes y)$. 

\begin{remark}
 In fact, the operators ${\cal I}_F$ and ${\cal I}_F$ are {\it complete} isometric isomorphisms. 
As to ${\cal I}_F$, one can prove this in the following way. First, we identify, up to a natural 
complete isometric isomorphism, $\cK[\mathcal{CB}(E\ot_{pop}F,G)]$ with 
$\mathcal{CB}(E\ot_{pop}F,\cK G)$ and $\cK[\CB(F,\CB(E,G))]$ with $\CB(F,\CB(E,\cK G))$, and then 
apply Theorem 7.4 to the triple $E,F,\cK G$. Here $\cK G$ is equipped with a $PQ$-norm by means of 
the embedding of $\cK[\cK G]$ into $\cK G$, well defined by taking $a[bz]$ to $(a\di b)z; 
a,b\in\cK, z\in G$. The details are given in\cite[Ch.8.8]{heb2} in the context of $Q$--spaces, but 
the argument is valid, up to some minor modifications, for $PQ$--spaces as well. 
\end{remark}

\section{Comparison of proto-operator-projective and operator-projective tensor products}
In conclusion, we consider the relationship between the introduced tensor product and the 
well-known operator-projective tensor product of operator spaces. We recall that the latter 
linearizes completely bounded bilinear operators within the class of $Q$--spaces which is 
essentially narrower than the class of $PQ$--spaces. Its initial definition was given in terms of 
an explicit construction (cf., e.g., the textbook~\cite[p. 124]{efr}), which after the translation 
from the `matrix' to the `non-coordinate' language sounds as follows. 

Let $E, F$ be linear spaces. It is not difficult to show that every $U\in\cK(E\ot F)$ can be 
expressed with the help of a `single diamond', namely as $a\cd(u\di v)\cd b; a,b\in\cK, u\in\cK E, 
v\in\cK F$. Thus, we can introduce the number 
\begin{align}
\|U\|_{op}=\inf\{\|a\|\|u\|\|v\|\|b\|\}.
\end{align}
where the infimum is taken over all possible representations of $U$ in the form $a\cd(u\di v)\cd 
b$. 

If $E, F$ are $Q$--spaces, then $\|\cd\|_{op}$ is an $Q$--norm on $E\ot F$, and the $Q$--space 
$E\ot_{op}F:=(E\ot F,\|\cd\|_{op})$, together with $\vartheta$, possess the universal property, 
characteristic for an operator-projective tensor product (see, e.g.,~\cite[Ch.7.2]{heb2}). 

It is known that within the class of quantum spaces the norm $\|U\|_{op}$ coincides with the norm 
$\|U\|_{pop}$, given by the formula (5.3) ({\it ibidem}). The difference lies in another corner: 
outside this class the former number is, generally speaking, essentially greater than the latter 
number. This can be demonstrated by the following example. 

\medskip
Set $E:=F:=\ell_1$, where by $\ell_1$ we denote, for brevity, the $PQ$--space 
$L_1(\N,{^{(\ii)}}\co)$ with the counting measure on $\N$, which is a particular case of the 
$PQ$--spaces $L_p(\cd)$ from Section 3.

Denote by $e_k\in\ell_1; k=1,2,\dots$ the sequence
$(...,0,1,0,...)$ with 1 as the $k-th$ coordinate, fix $n\in\N$ and choose arbitrary pairwise 
orthogonal rank 1 projections $P_k\in\cK; k=1,...,n$, in $\cK$. Finally, take in 
$\cK(\ell_1\ot_{pop}\ell_1)$ the element 
$$
V_n:=\sum_{k=1}^nP_k(e_k\ot e_k).
$$
\begin{proposition}
We have $\|V_n\|_{pop}=n$, whereas $\|V_n\|_{op}=n^2$. 
\end{proposition}
\begin{proof}
 To show the first equality, we use Theorem 6.9. As a particular case, it provides a completely 
 isometric isomorphism 
 $I:\ell_1\widehat\ot_{pop}\ell_1\to L_1(\N\times\N,{^{(\ii)}}\co\widehat\ot_{pop}{^{(\ii)}}\co)$.
 Clearly, the latter $PQ$--space can be identified with $L_1(\N\times\N,{^{(\ii)}}\co)$.
Thus, we can say that  $I_\ii(V_n)=\sum_{k=1}^nP_k{\bar e_k}$, where ${\bar e_k}$ is the function ( 
= double sequence) taking $(k,k)$ to 1 and taking other pairs in $\N\times\N$ to 0. Therefore the 
definition of $PQ$--norm on \\ $L_1(\N\times\N,{^{(\ii)}}\co)$ implies that $\|V_n\|_{pop}$ is the 
norm of the $\cK$-valued function in $L_1(\N\times\N,\cK)$, taking $(k,k)$ to $P_k$ in the case, 
when $1\le k\le n$ and taking other pairs in $\N\times\N$ to 0. This norm is, of course, 
$\sum_{k=1}^n\|P_k\|=n$. 

\medskip
Turn to the second equality. To begin with, we shall display the representation of $V_n$ in the 
form (8.1), such that $\|a\|\|u\|\|v\|\|b\|=n^2$. 

Take $\widetilde e_k\in L$ such that for every $k$ we have $P_k=\widetilde e_k\circ\widetilde e_k$. 
Further, 
take $u=v=\sum_{k=1}^nP_ke_k$, $a=\sum_{k=1}^n\widetilde e_k\circ(\widetilde e_k\di\widetilde e_k)$ 
and $b=\sum_{k=1}^n(\widetilde e_k\di\widetilde e_k)\circ\widetilde e_k$. A simple calculation 
shows that indeed $a\cd(u\di v)\cd b=V_n$, and it remains to observe that $\|a\|=\|b\|=1$ whereas 
$\|u\|=\|v\|=n$. 

Now we must show that for every representation of $V_n$ as $a\cd(u\di v)\cd b$ we have 
$\|a\|\|u\|\|v\|\|b\|\ge n^2$. 

Since $u,v\in\cK\ell_1$, it is easy to observe that $u$ can be represented as 
$u=\sum_{k=1}^na_ke_n+\sum_{l=1}^{m_1}a_l'f_l$ for some $m_1$, and $v$ can be represented as 
$v=\sum_{k=1}^nb_ke_n+\sum_{l=1}^{m_2}b_l'g_l$ for some $m_2$, where $a_k,a_l',b_k,b_l'\in\cK$, the 
sequences $f_l,g_l\in\ell_1$ begin with $n$ zeroes and the systems $f_l;l=1,...,m_1$ and 
$g_l;l=1,...,m_2$ are linearly independent. Consequently, we have $V_n=a\cd W\cd b$, where $W$ is
\[
\sum_{k=1}^n\sum_{l=1}^n(a_k\di b_l)(e_k\ot e_l)+
\sum_{k=1}^n\sum_{i=1}^{m_2}(a_k\di b_i')(e_k\ot g_i)+
\]
\[
\sum_{k=1}^n\sum_{j=1}^{m_1}(a_j'\di e_k)(f_j\ot e_k)+
\sum_{j=1}^{m_1}\sum_{i=1}^{m_2}(a_j'\di b_i')(f_j\ot g_i).
\]
But at the same time $V_n$ is $\sum_{k=1}^nP_k(e_k\ot e_k)$ and the system of elements in 
$\ell_1\ot\ell_1$, consisting of all $e_k\ot e_l, e_k\ot g_i,f_j\ot e_k$ and $f_j\ot g_i$ is 
obviously linearly independent. Therefore, comparing both representations of $V_n$, we see that 
$a(a_k\di b_k)b=P_k;k=1,\dots,n$, and all operators $a(a_k\di b_l)b$, where $k\ne l$, as well as 
all $a(a_k\di b_i')b, a(a_j'\di e_k)b, a(a_j'\di b_i')b$, are zeroes. In particular, we have 
$$
1=\|P_k\|=\|a(a_k\di b_k)b\|\le\|a\|\|a_k\di b_k\|\|b\|=\|a\|\|a_k\|\|b_k\|\|b\|
$$
for every $k$. 

Embedding $\cK\ell_1$ into $\ell_1(\cK)$ by the recipe in Section 4, we see that the indicated 
forms of $u$ and $v$ imply that $\|u\|\ge\sum_{k=1}^n\|a_k\|$ and $\|v\|\ge\sum_{k=1}^n\|b_k\|$. 
Set $\lm_k:=\|a\|\|a_k\|$ and note that $\|b_k\|\|b\|\ge\lm^{-1}$. Therefore we have 
$$
\|a\|\|u\|\|v\|\|b\|\ge\left(\sum_{k=1}^n\lm_k\right)\left(\sum_{k=1}^n\lm_k^{-1}\right)=
\sum_{k,l=1}^n\lm_k\lm_l^{-l};
$$
 the obvious estimate $\lm_k\lm_l^{-l}+\lm_k^{-1}\lm_l\ge2$ implies that the latter sum is $\ge n^2$. 
 \end{proof}

\medskip
{\rm This proposition shows, in particular, that outside the class of $Q$--spaces the function 
$U\mt\|U\|_{op}; U\in\cK(E\ot F)$ is not bound to be a norm. Indeed, let $m$ be a natural number 
such that $1\le m<n$. Set $V:=V_m, W:=V_n-V_m$. Then practically the same argument shows that 
$\|W\|_{op}=(n-m)^2$, 
 hence contrary to the triangle inequality, we have
 $$
 \|V+W\|_{op}=\|V_n\|_{op}=n^2>m^2+(n-m)^2=\|V\|_{op}+\|W\|_{op}.
$$

\bigskip
It is a pleasure for the author to thank the referee, whose questions and comments 
contributed to the improvement and enrichment of the present paper. 

 }

\begin{flushleft}
Moscow State (Lomonosov) University\\ Moscow, 111991, Russia\\
E-mail address: helemskii@rambler.ru 
\end{flushleft}


\begin{thebibliography}{999}

\bibitem{ble}
D.~P.~Blecher, The standard dual of an operator space, Pacific J. Math., No. 1, 153 (1992), 15-30. 
\bibitem{blem}
D.~P.~Blecher, C.~Le~Merdy, {\it Operator Algebras and their Modules},  Clarendon Press, Oxford, 
2004. 
\bibitem{blp}
D.~P.~Blecher, V.~I.~Paulsen, Tensor products of operator spaces, J. Funct. Anal.,  No. 2, 99 
(1991), 262-292 . \bibitem{bou} N.~Bourbaki, {\it Alg\`ebre, Chap.III: Alg\`ebre multilin\'eaire}, 
Hermann, Paris, 1948. 
\bibitem{dlot}
H.~G.~Dales, N.~J.~Laustsen, T.~Oikhberg, V.~G.~Troitsky, Multi-norms and Banach lattices. To 
appear in Dissertationes Math. 
\bibitem{ef5}
E.~G.~Effros, Advances in quantized functional analysis, Proc. ICM Berkeley, 1986. 
\bibitem{er1}
E.~G.~Effros, Z.-J.~Ruan, On matricially normed spaces, Pacific J. Math, 132, (1988) 243-64. 
\bibitem{ro}
E.~G.~Effros, Z.-J.~Ruan, Representations of operator bimodules and their applications, J. Operator 
Theory, 19 (1988), 137-157. 
\bibitem{er2}
E.~G.~Effros, Z.-J.~Ruan, A new approach to operator spaces, Canad. Math. Bull., 34 (1991), 
329-337. 
\bibitem{er5}
E.~G.~Effros, Z.-J.~Ruan, On the abstract characterization of operator spaces, Proc. Amer. Math. 
Soc., 119 (1993), 579-584. 
\bibitem{efr}
E.~G.~Effros, Z.-J.~Ruan, {\it Operator Spaces},  Clarendon Press, Oxford, 2000. 
\bibitem{gro}
A.~Grothendieck, {\it Produits Tensoriels Topologiques et Espaces Nucleaires}, Mem. Amer. Math. 
Soc., No. 16, 1955. 
\bibitem{he6}
A.~Ya.~Helemskii, Tensor products in quantum functional analysis: non-coordinate approach. In: 
`{\it Topological Algebras and Applications}', A.~Mallios, M.~Haralampidou, Eds. American 
Mathematical Society, Providence, R.I. (2007), 199-224. 
\bibitem{heb2}
A.~Ya.~Helemskii, {\it Quantum Functional Analysis.} American Mathematical Society, Providence, 
R.I., 2010. 
\bibitem{her}
A.~Ya.~Helemskii, Extreme flatness of normed modules and Arveson-Wittstock type theorems, J. 
Operator Theory, 64:1 (2010), 101-112. 
\bibitem{hend}
A.~Ya.~Helemskii, Multi-normed spaces, based on non-discrete measures, and their tensor products.
Submitted to Izvestiya RAN, ser. math. 
\bibitem{lam}
 A.~Lambert, {\it Operatorfolgenr\"aume}. Dissertation. Saarbr\"uken. 2002.
\bibitem{mcl}
S.~ Mac Lane, {\it Categories for the Working Mathematician}. Springer-Verlag, Berlin, 1971. 
\bibitem{paul}
 V.~I.~Paulsen. {\it Completely Bounded Maps and Operator Algebras}, Cam. Univ. Press, Cambridge,
2002. 
\bibitem{pis}
J.~Pisier, {\it Introduction to Operator Space Theory}, Cam. Univ. Press, Cambridge, 2003. 
\bibitem{ru2}
Z.-J.~Ruan, Subspaces of $C^*$-algebras, J. Funct. Anal., 76 (1988), 217-230. 
\bibitem{tom}
J.~Tomiyama, On the transpose map of matrix algebras, Proc. Amer. Math. Soc., 88 (1983), 635-638. 
\bibitem{wit}
G.~Wittstock, Injectivity of the module tensor product of semi-Ruan modules, J. Operator Theory, 
65: 1 (2010), 87-113. 

\end{thebibliography}
\end{document}